%% file: main.tex
\documentclass{mpi2015-cscpreprint}

\usepackage[american]{babel}
\usepackage{graphicx,subfig,float,tikzscale,tikz,pgfplots}
\usepackage{amssymb}
\usepackage{amsthm}

\DeclareMathOperator*{\esssup}{ess\,sup}
\DeclareMathOperator*{\essinf}{ess\,inf}

\newtheorem{theorem}{Theorem}[section]

\newtheorem{lemma}[theorem]{Lemma}

\newtheorem{problem}{Problem}[section]
\newtheorem{definition}{Definition}[section]
\newtheorem{remark}{Remark}[section]

\numberwithin{equation}{section}

\begin{document}

\title{Periodic optimal control \\of a plug flow reactor model \\with an isoperimetric constraint}
  
\author[$1$]{Yevgeniia~Yevgenieva}
\affil[$1$]{Max Planck Institute for Dynamics of Complex Technical Systems,
              39106 Magdeburg, Germany\\ \newline
              Institute of Applied Mathematics and Mechanics, National Academy of Sciences of Ukraine, \newline 84116 Sloviansk, Ukraine\authorcr
  \email{yevgenieva@mpi-magdeburg.mpg.de}, \orcid{0000-0000-0000-0000}}
  
\author[$2$]{Alexander~Zuyev}
\affil[$2$]{Max Planck Institute for Dynamics of Complex Technical Systems,
              39106 Magdeburg, Germany\\ \newline
              Otto von Guericke University Magdeburg, 39106 Magdeburg, Germany\\ \newline
              Institute of Applied Mathematics and Mechanics, National Academy of Sciences of Ukraine, \newline 84116 Sloviansk, Ukraine\authorcr
  \email{zuyev@mpi-magdeburg.mpg.de}, \orcid{0000-0000-0000-0000}}

\author[$3$]{\, Peter~Benner}
\affil[$3$]{Max Planck Institute for Dynamics of Complex Technical Systems,
              39106 Magdeburg, Germany\\ \newline
              Otto von Guericke University Magdeburg, 39106 Magdeburg, Germany\authorcr
  \email{benner@mpi-magdeburg.mpg.de}, \orcid{0000-0000-0000-0000}}

\author[$4$]{Andreas~Seidel-Morgenstern}
\affil[$4$]{Max Planck Institute for Dynamics of Complex Technical Systems,
              39106 Magdeburg, Germany\\ \newline
              Otto von Guericke University Magdeburg, 39106 Magdeburg, Germany\authorcr
  \email{anseidel@ovgu.de}, \orcid{0000-0000-0000-0000}}
  
\shorttitle{Periodic optimal control...}
\shortauthor{Ye.~Yevgenieva, A.~Zuyev, P.~Benner, A.~Seidel-Morgenstern}
\shortdate{}
  

  
\abstract{We study a class of nonlinear hyperbolic partial differential equations with boundary control. This class describes chemical reactions of the type ``$A \to$ product'' carried out in a plug flow reactor (PFR)  in the presence of an inert component. An isoperimetric optimal control problem with periodic boundary conditions and input constraints is formulated for the considered mathematical model in order to maximize the mean amount of product over the period. 
For the single-input system, the optimality of a bang-bang control strategy is proved in the class of bounded measurable inputs. The case of controlled flow rate input is also analyzed by exploiting the method of characteristics. A case study is performed to illustrate the performance of the reaction model under different control strategies.}

 \novelty{The key contributions of our work are the following:
\begin{itemize}
\item an analytic representation of the cost functional is derived for the PFR model in the cases of one- and two-dimensional inputs by using the method of characteristics;
\item local conditions are obtained in the general class of measurable control functions;
\item the optimal controls are not unique, and a parameterization with one switching only can be used to achieve the optimality condition.
\end{itemize}}

\maketitle

\section{Introduction}

It has been known for several decades that periodic control strategies can improve the performance of nonlinear chemical reactions in comparison to their steady-state operations~\cite{D67,WOM81,SH13}.
On the one hand, lumped parameter reaction models with harmonic inputs have been extensively studied in the literature with the use of frequency-domain methods (see, e.g.,~\cite{F21,N22} and references therein),
On the other hand, it follows from the Pontryagin maximum principle that the optimal controls are bang-bang for the maximization of the average reaction product within the considered class of models~\cite{S76,WOM81,ZS-MB2017}. An analytical design of periodic bang-bang controllers has been proposed in~\cite{BSZ19} for the isoperimetric optimization problem.
The above optimal control techniques, developed for model systems or ordinary differential equations, are not directly applicable to infinite-dimensional reaction models. 

An important class of distributed parameter control systems is represented by mathematical models of plug flow reactors (PFR) governed by hyperbolic systems of partial differential equations~\cite{BC16}.
Even though there is a comprehensive engineering literature on PFR models (cf.~\cite{SH13} and references therein), the periodic optimal control problems require a rigorous analysis from the mathematical viewpoint. Just a few results, dealing with non-optimality of steady state solutions and comparison of different control strategies~\cite{Grab1985,S91} as well as the $\Pi$-test and properness condition~\cite{NS97}, are available in this area.

In this paper, we will study the nonlinear hyperbolic control systems that describe chemical reactions of the type ``$A\rightarrow$ product'' carried out in a PFR  in the presence of an additional inert component (dilutant or solvent).
The key contributions of our work are summarized below:
\begin{itemize}
\item an analytic representation of the cost functional is derived for the PFR model in the cases of one- and two-dimensional inputs by using the method of characteristics;
\item local optimality conditions are obtained in the general class of measurable control functions;
\item the optimal controls are not unique, and a parameterization with one switching only can be used to achieve the local optimality condition.
\end{itemize}

The rest of this paper is organized as follows. A single-input nonlinear control system will be considered in Section~\ref{PFR} as a PFR model with boundary injection. The isoperimetric optimal control problem will be  solved for this model in order to maximize the conversion of the input reactant ($A$) into the product. An extension of these results to the PFR with time-varying flow-rate will be presented in Section~\ref{PFR2}.
A comparative analysis of different control strategies will be performed in Section~\ref{case_study} under a specific choice of reaction parameters.
Finally, Section~\ref{conc} contains some concluding remarks.

\section{Plug flow reactor model}\label{PFR}

Consider an isothermal reaction of the type ``$A\rightarrow$ product'' in a plug flow reactor (PFR) model~\cite[p.~394]{SH13}:
\begin{equation}\label{eq1}
\frac{\partial C_A (x,t)}{\partial t}+v\frac{\partial C_A(x,t)}{\partial x}=-kC_A^n(x,t),\qquad (x,t)\in \Omega=[0,L]\times\mathbb{R},
\end{equation}
\begin{equation}\label{bound1}
C_A(0,t)=C_{A_0}(t),
\end{equation}
where $C_A(x,t)$ is the reactant $A$ concentration inside the reactor at the distance $x$ from the inlet and time $t$, $L$ is the length of the reactor tube, $C_{A_0}(t)$ is the concentration of $A$ in the inlet stream that contains also another inert component, $n>0$ is the reaction order, $v>0$ is the flow-rate of the reaction stream, and $k>0$ is the kinetic constant. The function $C_{A_0}(t)\in [C_{min},C_{max}]$ is treated as the control input and assumed to be bounded by some constants $C_{max}>C_{min}>0$.

The boundary value problem~\eqref{eq1}, \eqref{bound1} can be solved by the method of characteristics~\cite{Grab1985}:
\begin{equation}\label{solution1}
C_A(x,t)=\left(C_{A_0}\left(t-\frac xv\right)^{-(n-1)}+\frac{k(n-1)}{v}x\right)^{-\frac1{n-1}},\quad n\neq1.
\end{equation}
For the case $n=1$, the solution has the following form:
\begin{equation}\label{solution11}
C_A(x,t)=C_{A_0}\left(t-\frac xv\right)e^{-\frac kvx}.
\end{equation}

If the function $C_{A_0}(t)$ is continuously differentiable on $\mathbb R$, then expressions~\eqref{solution1}--\eqref{solution11} define the classical solution of the problem~\eqref{eq1}--\eqref{bound1}.
It is easy to see that, in order to define $C(x,s)$ for all $x\in [0,L]$ at a given $s$, the information about $C_{A_0}(t)$ on the closed interval $t\in [s-\frac{L}{v},s]$ is needed. We are interested in studying an optimal control problem for system~\eqref{eq1}--\eqref{bound1} with  $\tau$-periodic controls $C_{A_0}(t)$. In this case, it suffices to define the control $C_{A_0}(t)$ on an interval $t\in [0,\tau)$  and extend it to $t\in\mathbb R$ by $\tau$-periodicity.
For the subsequent formal analysis,
we allow the functions $C_{A_0}(t)$ to be discontinuous and introduce the class of admissible controls ${\cal U}_\tau$ as follows.

\begin{definition}
Let $\tau>0$, $C_{max}>C_{min}>0$, and $\overline{C}\in [C_{min},C_{max}]$ be given.
The class of admissible controls $\mathcal{U}_{\,\tau}$
consists of all locally measurable functions $C_{A_0}:\mathbb R\to [C_{min},C_{max}]$ such that $C_{A_0}(t)$ is $\tau$-periodic and
\begin{equation}\label{isoperim1}
\frac1\tau\int_0^\tau C_{A_0}(t)dt=\overline{C}. 
\end{equation}
\end{definition}

Formulas~\eqref{solution1}--\eqref{solution11} correctly define the function $C:\Omega \to \mathbb R$ for any $C_{A_0}\in {\cal U}_\tau$. We will refer to these functions $C(x,t)$ as weak solutions of the problem~\eqref{eq1}--\eqref{bound1} (see, e.g.,~\cite{Kruzhkov}). Indeed, the above defined $C(x,t)$ satisfies the integral identity 
\begin{equation}\label{ii1}
\int_\Omega \left(C_A\frac{\partial\varphi}{\partial t}+vC_A\frac{\partial\varphi}{\partial x}+kC_A^n\varphi\right)dx\,dt=0,
\end{equation}
for each smooth test function $\varphi\in C_0^\infty(\Omega)$ with compact support.

Our goal is to optimize the conversion of $A$ to the product  by using time-varying inputs $C_{A_0}(t)$ under the isoperimetric constraint~\eqref{isoperim1} over a given period $\tau$ as follows.

\begin{problem}\label{prob1}
Given $\tau>0$ and $\overline{C}\in[C_{min},C_{max}]$, find a control  $\hat C_{A_0}(\cdot)\in \mathcal{U}_{\,\tau}$ that minimizes the cost
\begin{equation}\label{cost1}
J[C_{A_0}]=\frac1\tau\int_0^\tau C_A(L,t)\,v\,dt
\end{equation}
among all admissible controls $C_{A_0}(\cdot)\in \mathcal{U}_{\,\tau}$. This cost function evaluates the mean molar flux of component $A$ that leaves the reactor divided by the cross section area of the tube (in $[mol\,s^{-1}\,m^{-2}]$). Here, the right-hand side of~\eqref{cost1} contains the (weak) solution $C_A(x,t)$ of the problem \eqref{eq1}, \eqref{bound1} corresponding to the control $C_{A_0}(t)$, so $J[C_{A_0}]$ is well-defined in terms of $C_{A_0}$ by formulas~\eqref{solution1},~\eqref{solution11}.                                     
\end{problem}

In order to describe the optimal controls for Problem~\ref{prob1}, we use the following notations.
For a function $u:[0,\tau)\to \mathbb R$, its $\tau$-periodic extension is denoted by $u^\tau:\mathbb R \to \mathbb R$, so that
$u^\tau(t)\equiv u(t)$ for $t\in [0,\tau)$, and the function $u^\tau(t)$ is $\tau$-periodic.
The Lebesgue measure of a set $A\subset \mathbb R$ is denoted by $\mu(A)$. Now we formulate the main result of this section.

\begin{theorem}\label{th1} Let $\tau>0$, $C_{max}>C_{min}>0$, and $\overline{C}\in [C_{min},C_{max}]$ be given.
\begin{itemize}
\item[1)] If $n=1$, then all control functions $C_{A_0}\in\mathcal{U}_{\,\tau}$ give the same value for the cost functional $J$.
\item[2)] If $n<1$ and $C_{min}>\left(\frac{v}{kL(1-n)}\right)^{-\frac1{1-n}}$, then the steady-state control $C_{A_0}(t)=\overline{C}$ is optimal for Problem~\ref{prob1}.
\item[3)]If $n>1$, then the piecewise constant control $C_{A_0}(t)=u^\tau(t)$ is optimal for Problem~\ref{prob1}, where
\begin{equation}\label{bang1}
u(t)=
\begin{cases}
C_{min}, &\text{ if }\;t\in A^-,
\\ C_{max}, &\text{ if }\;t\in A^+=[0,\tau)\setminus A^-,
\end{cases}
\end{equation}
and $A^-\subset [0,\tau)$ is any Lebesgue-measurable set such that 
\begin{equation*}
\mu(A^-)=\frac{C_{max}-\overline{C}}{C_{max}-C_{min}}\tau.
\end{equation*}
\end{itemize}
\end{theorem}

Note that, as each admissible control $C_{A_0}\in {\cal U}_\tau$ is periodic,
the corresponding function $C_A(x,t)$ in \eqref{solution1} and \eqref{solution11} is also $\tau$-periodic.
For the case $n\neq1$, we modify the cost functional due to periodicity as follows:
\begin{multline}\label{cost1*}
J=J[C_{A_0}]= \frac v\tau\int_0^\tau \left(C_{A_0}\left(t-Lv^{-1}\right)^{-(n-1)}+\frac{kL(n-1)}{v}\right)^{-\frac1{n-1}}dt
\\= \frac v\tau\int_0^\tau \left(C_{A_0}(t)^{-(n-1)}+\frac{kL(n-1)}{v}\right)^{-\frac1{n-1}}dt=:\frac v\tau\int_0^\tau \Phi(C_{A_0}(t)) dt.
\end{multline}

It is easy to see that the function $\Phi$ is increasing and concave if $n>1$. Indeed,
\begin{align}
&\Phi'(\xi)=\left[1+\frac{kL(n-1)}{v}\xi^{n-1}\right]^{-\frac n{n-1}}>0\;\text{ if }n>1,\label{Phi'}
\\&\Phi''(\xi)=-\frac{kL(n-1)n}{v}\left[1+\frac{kL(n-1)}{v}\xi^{n-1}\right]^{-\frac{2n-1}{n-1}}\xi^{n-2}<0\;\text{ if }n>1.\label{Phi''}
\end{align}

To prove Theorem~\ref{th1}, we need to define a special class of control functions for Problem~\ref{prob1}.

\begin{definition}\label{def3}
A function $c:{\mathbb R}\to [C_{min},C_{max}]$ belongs to the class $\mathcal{A}_{\,\widetilde{C}}$ for a given constant $\widetilde{C}\in [C_{min},C_{max}]$, if  $c(\cdot)\in {\cal U}_\tau$ and there exist Lebesgue-measurable sets $A^+\subset [0,\tau)$,  $A^-\subset [0,\tau)$ such that:
\begin{itemize}
    \item[1)] $\essinf_{t\in\mathbb{A}^+}c(t)\geqslant\widetilde{C}$;

    \item[2)] $\esssup_{t\in\mathbb{A}^-}c(t)\leqslant\widetilde{C}$;

    \item[3)] $\mu(A^+\cap A^-)=0$, $\mu(A^+\cup A^-)=\tau$; 
    
    \item[4)] $\mu(A^-)=\frac{C_{max}-\overline{C}}{C_{max}-C_{min}}\tau$.
    
\end{itemize}
\end{definition}

In the paper~\cite{MS-MP2008}, it was reported that the sinusoidal inputs ensure a better performance
of the PFR reactor (with respect to the cost $J$) in comparison to the steady-state input $C_{A_0}(t)\equiv \overline{C}$ if  $n>1$.
We will show in the lemma below that the bang-bang strategies have even better performance than the sinusoidal ones. It is easy to see that $\mathcal{A}_{\tilde C}$ contains the sinusoidal functions. For instance, assuming that $\overline{C}-C_{min}=C_{max}-\overline{C}$, one can show that the function
\begin{equation*}
C(t)=\overline{C}+(\overline{C}-C_{min})\sin\left(\frac{2\pi}{\tau}t\right)
\end{equation*}
belongs to the class $\mathcal{A}_{\,\overline{C}}$. Indeed, setting $A^+=[0,\tau/2)$ and $A^-=[\tau/2,\tau)$, we meet all the requirements of Definition~\ref{def3}.


\begin{lemma}\label{lem1}
Let $n>1$ and let $C_\mathcal{A}(\cdot)\in \mathcal{A}_{\,\widetilde{C}}$ for some $\widetilde{C}\in(C_{min},C_{max})$. Then there exists a control $C_b(\cdot)\in {\cal U}_\tau$ such that
$$
J [C_b] \leqslant J [C_\mathcal{A}],
$$
where the cost $J$ is defined in Problem~\ref{prob1}.
\end{lemma}

\begin{proof}
Consider an arbitrary function $C_\mathcal{A}(t)$ from the class $\mathcal{A}_{\,\widetilde{C}}$ with a fixed $\widetilde{C}\in(C_{min},C_{max})$ and define now the corresponding bang-bang control $c_b:[0,\tau)\to [C_{min},C_{max}]$:
\begin{equation}\label{bang3}
c_{b}(t)=
\begin{cases}
C_{min}, &\text{ if }t\in A^-,
\\ C_{max}, &\text{ if }t\in A^+,
\end{cases}
\end{equation}
where  the sets $A^+$, $A^-$ correspond to the class $\mathcal{A}_{\widetilde{C}}$. Using condition 4) from Definition~\ref{def3}, we conclude that the isoperimetric condition~\eqref{isoperim1} holds for the function $C_b=c_b^\tau$ (the $\tau$-periodic extension of $c_b$), so $C_b\in\mathcal{U}_{\,\tau}$.

Now using the isoperimetric condition \eqref{isoperim1} and the property of concave functions:
\begin{equation*}
\Phi(x)-\Phi(y)\leqslant  \Phi'(y)(x-y)\quad\forall\,x,y\in[C_{min},C_{max}],
\end{equation*}
we investigate the difference of costs $J[C_b]-J[C_\mathcal{A}]$:
\begin{equation*}
\begin{aligned}
&J[C_b]-J[C_\mathcal{A}]=\frac v\tau\int_0^\tau \left(\Phi(C_b(t))-\Phi(C_\mathcal{A}(t))\right)dt 
\\&\leqslant\frac v\tau\int_0^\tau\Phi'(C_\mathcal{A}(t))\left(C_b(t)-C_\mathcal{A}(t)\right)dt
\\&=\frac v\tau\int_{A^+}\Phi'(C_\mathcal{A}(t))\left(C_{max}-C_\mathcal{A}(t)\right)dt+\frac v\tau\int_{A^-}\Phi'(C_\mathcal{A}(t))\left(C_{min}-C_\mathcal{A}(t)\right)dt
\\&\leqslant\frac v\tau\int_{A^+}\Phi'(\widetilde{C})\left(C_{max}-C_\mathcal{A}(t)\right)dt+\frac v\tau\int_{A^-}\Phi'(\widetilde{C})\left(C_{min}-C_\mathcal{A}(t)\right)dt
\\&=\frac v\tau\Phi'(\widetilde{C})\int_0^\tau\left(C_{b}(t)-C_\mathcal{A}(t)\right)dt=0.
\end{aligned}
\end{equation*}
The obtained estimate proves Lemma~\ref{lem1}.
\end{proof}

\begin{lemma}\label{lem2}
For any function $u(\cdot)\in\mathcal{U}_{\,\tau}$, there exists a constant \\$\widetilde{C}\in(C_{min},C_{max})$ such that   $u(\cdot)\in\mathcal{A}_{\,\widetilde{C}}$.
\end{lemma}

\begin{proof}
Denote the values 
\begin{equation*}
\mu_{+}:=\frac{\overline{C}-C_{min}}{C_{max}-C_{min}}\tau,
\qquad \mu_{-}:=\frac{C_{max}-\overline{C}}{C_{max}-C_{min}}\tau.
\end{equation*}
It is clear that $\mu_{+}+\mu_{-}=\tau$.

Assume that there exists a function $u\in\mathcal{U}_{\,\tau}$ which does not belong to any class $\mathcal{A}_{\,\widetilde{C}}$.  Due to Definition~\ref{def3}, this means that, for any $\widetilde{C}\in(C_{min},C_{max})$, either
\begin{equation}\label{mes+}
    \mu(\{t\in[0,\tau):u(t)>\widetilde{C}\})>\mu_{+},
\end{equation}
or
\begin{equation}\label{mes-}
    \mu(\{t\in[0,\tau):u(t)<\widetilde{C}\})>\mu_{-}.
\end{equation}

Consider the case \eqref{mes+} and rewrite this statement in the following way. For any arbitrary small $\delta>0$, the following inequality always holds:
\begin{equation}\label{mes+_delta}
    \mu(A^+_\delta)>\mu_{+},\text{ where }A^+_\delta:=\{t\in[0,\tau):u(t)\geqslant C_{max}-\delta\}.
\end{equation}
It is easy to prove that the only function which can satisfy the above statement is constant, namely, $u(t)=C_{max}$ for $\mu$-a.a. $t\in A^+_\delta$. Now we calculate the mean value of the function $u$:
\begin{align*}
&\frac 1\tau\int_0^\tau u(t)\,dt=\frac 1\tau\int_{A^+_\delta} u(t)\,dt+\frac 1\tau\int_{[0,\tau)\setminus A^+_\delta} u(t)\,dt
\\&\geqslant 
C_{max}\frac{|A^+_\delta|}\tau+C_{min}\frac{\tau-|A^+_\delta|}\tau
=(C_{max}-C_{min})\frac{|A^+_\delta|}\tau+C_{min}
\\&>(C_{max}-C_{min})\frac{\mu_{+}}\tau+C_{min}=\overline{C}.
\end{align*}
Thus we get that the isoperimetric constraint \eqref{isoperim1} is violated, so $u\notin\mathcal{U}_{\,\tau}$, which contradicts our assumption.

Using the same arguments, one can prove that $\frac 1\tau\int_0^\tau u(t)\,dt<\overline{C}$ in the case \eqref{mes-}.
\end{proof}

\begin{proof}[Proof of  Theorem~\ref{th1}.]

For the case $n=1$, evaluating directly the value of the cost functional for solution \eqref{solution11}, we get:
\begin{equation*}
\begin{aligned}
J&=\frac v\tau\int_0^\tau C_A(L,t)dt= \frac v\tau\int_0^\tau C_{A_0}\left(t-\frac Lv\right)e^{-\frac kvL}dt
\\&=e^{-\frac kvL}\frac v\tau\int_0^\tau C_{A_0}(t)dt=
\frac v\tau e^{-\frac{kL}v}\overline{C}
\end{aligned}
\end{equation*}
for any admissible $\tau$-periodic control function $C_{A_0}(t)$.

In the case $n<1$, it follows from \eqref{Phi''} that $\Phi$ is a convex function, provided that $C_{min}>\left(\frac{v}{kL(1-n)}\right)^{-\frac1{1-n}}$. Using Jensen's inequality for convex functions, we get
\begin{equation*}
\Phi(\overline{C})<\frac1\tau\int_0^\tau \Phi(C_{A_0}(t)) dt
\end{equation*}
for each non-negative Lebesgue-integrable function $C_{A_0}$, which proves the second statement of the theorem.

For the case $n>1$, we have the opposite Jensen's inequality which means that any non-negative Lebesgue--integrable periodic function $C_{A_0}$, which satisfied the constraints of Problem~\ref{prob1}, ensures a better performance in the sense of the functional $J$ in comparison with the steady-state control $\overline{C}$ (see also \cite{Grab1985}).  Due to Lemmas~\ref{lem1}, \ref{lem2}, the bang-bang strategy is the optimal control in terms of Problem~\ref{prob1}. 
\end{proof}

\begin{remark}\label{freq1}
We note that the number of switchings of control function \eqref{bang1} is not important in the case $n>1$, as the function $\Phi$ does not depend on time $t$ explicitly. So, there is a class of bang-bang controls which are equivalent in the sense of minimizing the cost functional $J[C_{A_0}]$. A~simple representative of this class is 
$C_{A_0}=c^\tau\in {\cal U}_\tau$ -- the $\tau$-periodic extension of the control $c:[0,\tau)\to [C_{min},C_{max}]$ with one switching of the following form:
\begin{equation*}
c(t)=
\begin{cases}
C_{min}, &\text{ if }t\in [0,\tau^*),
\\ C_{max}, &\text{ if }t\in [\tau^*,\tau),
\end{cases}
\end{equation*}
where $\tau^*=\frac{C_{max}-\overline{C}}{C_{max}-C_{min}}\tau$.
\end{remark}

\section{Plug flow reactor model considering a controlled flow-rate}\label{PFR2}

In this section, we investigate the mathematical model of PFR with a time-varying flow-rate $v(t)$:
\begin{equation}\label{eq2}
\frac{\partial C_A}{\partial t}+v(t)\frac{\partial C_A}{\partial x}+kC_A^n=0,\qquad (x,t)\in\Omega = [0,L]\times\mathbb{R},
\end{equation}
\begin{equation}\label{bound2}
C_A(0,t)=C_{A_0}(t).
\end{equation}
For this model, we are interested in the following optimal control problem: 
\begin{problem}\label{prob2}
Given positive constants $\tau$, $C_{min}<C_{max}$, $v_{min}<v_{max}$, $\overline{C}\in[C_{min},C_{max}]$, and 
$\overline{v}\in[v_{min},v_{max}]$, find $\tau$-periodic measurable controls $\hat C_{A_0}:\mathbb{R}\rightarrow[C_{min},C_{max}]$ and $\hat v:\mathbb{R}\rightarrow[v_{min},v_{max}]$ that minimize the cost
\begin{equation}\label{cost2}
J=\frac1\tau\int_0^\tau C_A(L,t)\,v(t)\,dt
\end{equation}
among all solutions $C_A(x,t)$ of the problem \eqref{eq2}, \eqref{bound2} corresponding to the class of admissible controls, i.e.
$\tau$-periodic measurable functions $C_{A_0}:\mathbb R\to [C_{min},C_{max}]$ and $v:\mathbb R\to [v_{min},v_{max}]$ that satisfy the isoperimetric constraint
\begin{equation}\label{isoperim2}
\frac1\tau\int_0^\tau C_{A_0}(t)v(t)dt=\overline{C}\,\overline{v}.
\end{equation}
\end{problem}
We will also consider the additional assumption:
\begin{equation}\label{assump2}
\int_0^\tau v(t)dt=L,
\end{equation}
which means that the residence time of the reaction is equal to $\tau$.
To ensure the isoperimetric condition \eqref{isoperim2}, we will assume that
\begin{equation}\label{assump2*}
C_{min}\leqslant\,\frac{\overline{C}\,\overline{v}\,\tau}{L}\leqslant C_{max}.
\end{equation}

We solve the problem \eqref{eq2}, \eqref{bound2} using the method of characteristics. Namely, we write the Lagrange equations to find the characteristics curves: 
\begin{align}
& \frac{dt}{ds}=1,\quad t(0,r)=r,\label{char2-1}
\\& \frac{dx}{ds}=v(t),\quad x(0,r)=0,\label{char2-2}
\\& \frac{dz}{ds}=-kz^n,\quad z(0,r)=C_{A_0}(r).\label{char2-3}
\end{align}
Here $t=t(s,r)$, $x=x(s,r)$ define the characteristic curve. 
Solving equations \eqref{char2-1} and \eqref{char2-2}, we get:
\begin{equation*}
\begin{aligned}
& t(s,r)=s+r,
\\& x(s,r)=V(s+r)-V(r),
\end{aligned}
\end{equation*}
where $V(t):=\int_0^tv(\xi)d\xi$ is a strictly increasing function due to the positivity of $v$. Thus, we can express $s$ and $r$ in terms of $x$ and $t$:
\begin{equation*}
\begin{aligned}
& s(x,t)=t-V^{-1}(V(t)-x),
\\& r(x,t)=V^{-1}(V(t)-x),
\end{aligned}
\end{equation*}
where the function $V^{-1}$ denotes the inverse to $V$.
Solving equation \eqref{char2-3} in the case $n\neq1$, we get the solution of the problem \eqref{eq2}, \eqref{bound2}:
\begin{equation}\label{solution2}
\begin{aligned}
z(s,r)&=\left[C_{A_0}(r)^{-(n-1)}+k(n-1)s\right]^{-\frac1{n-1}},
\\ C_A(x,t)&=\left[C_{A_0}\left(V^{-1}(V(t)-x)\right)^{-(n-1)}+k(n-1)\left(t-V^{-1}(V(t)-x)\right)\right]^{-\frac1{n-1}}.
\end{aligned}
\end{equation}
For the case $n=1$, the solution has the following form:
\begin{equation}\label{solution2*}
\begin{aligned}
& z(s,r)=C_{A_0}(r)e^{-ks},
\\& C_A(x,t)=C_{A_0}\left(V^{-1}(V(t)-x)\right)e^{-k\left(t-V^{-1}(V(t)-x)\right)}.
\end{aligned}
\end{equation}
Similarly to the previous study in Section~\ref{PFR}, we consider the obtained solutions as weak solutions of the problem~\eqref{eq2}, \eqref{bound2} from the class of measurable functions in the sense of the integral identity:
\begin{equation*}
\int_\Omega\left(C_A\frac{\partial\varphi}{\partial t}+vC_A\frac{\partial\varphi}{\partial x}+kC_A^n\varphi\right)dx\,dt=0,
\end{equation*}
for each smooth test function $\varphi\in C_0^\infty(\Omega)$ with compact support.

Now we give definitions of classes of admissible controls.

\begin{definition}
Let $\tau>0$, $C_{max}>C_{min}>0$, $v_{max}>v_{min}>0$, $L\in [v_{min}\tau,v_{max}\tau]$ and $\overline{C}\in [C_{min},C_{max}]$, $\overline{v}\in [v_{min},v_{max}]$ be given.
The class of admissible controls $\mathcal{V}_{\,\tau}$
consists of all locally measurable vector-functions $(c,v):\mathbb R\to [C_{min},C_{max}]\times[v_{min},v_{max}]$ such that $(c(t),v(t))$ is $\tau$-periodic and the isoperimetric conditions \eqref{isoperim2}, \eqref{assump2} hold.
\end{definition}


In solving the optimal control problem \ref{prob2}, we need the following definition of the class of bang-bang controls.

\begin{definition}\label{B_class}

The class $\mathcal{B}_{\tau}$ with given $\tau>0$ is a class of control vector-functions $(c,v)\in\mathcal{V}_{\,\tau}$ which define the bang-bang strategy with respect to the constraints of Problem~\ref{prob2}. Namely, the function $v$ has the following form:
\begin{equation}\label{v_opt}
v(t)=
\begin{cases}
v_{min}, &\text{ if }t\in B^-,
\\ v_{max} &\text{ if }t\in B^+,\text{ for } t\in [0,\tau),
\end{cases}
\end{equation}
where $B^-,B^+\subset [0,\tau)$, $\mu(B^+\cap B^-)=0$, $\mu(B^+\cup B^-)=\tau$, and $\mu(B^+)=\mu_v^+:=\frac{L-\tau v_{min}}{v_{max}-v_{min}}$, $\mu(B^-)=\mu_v^-:=\frac{\tau v_{max}-L}{v_{max}-v_{min}}$ are defined by assumption~\eqref{assump2}. 
Moreover, function $c$ is defined by the relation:
\begin{equation}\label{C_opt}
c(t)=
\begin{cases}
C_{max}, &\text{ if }t\in A^+,
\\ C_{min}, &\text{ if }t\in A^-,
\end{cases}
\end{equation}
where $A^+,A^-\subset [0,\tau)$, $\mu(A^+\cap A^-)=0$, $\mu(A^+\cup A^-)=\tau$, and $A^+\cap B^-$ has the maximum measure among all cases which are consisent with the isoperimetric constraint \eqref{isoperim2}. The measure of sets $A^\pm$ depends on the relation between the parameters of the problem. Namely, there are two possible cases:

\begin{itemize}
    \item[(i)] If 
\begin{align*}
\kappa:=&\,\tau v_{max}(\overline{C}\overline{v}-C_{max}v_{min})+\tau v_{min}(C_{min}v_{max}-\overline{C}\overline{v})+L(C_{max}v_{min}-C_{min}v_{max})\leqslant0
\end{align*}
then $\mu(A^+)\leqslant\mu(B^-)$ and $A^+\cap B^+=\emptyset$ (see Fig.1(a)). 
In this case, the measures of sets $A^\pm$ are as follows:
\begin{equation}
\mu(A^+)=\mu_c^+:=\frac{\tau\overline{C}\overline{v}-LC_{min}}{v_{min}(C_{max}-C_{min})},
\qquad \mu(A^-)=\mu_c^-:=\tau-\mu_c^+.
\end{equation}

    \item[(ii)] If $\kappa>0$ then $\mu(A^-)\leq\mu(B^+)$ and $A^-\cap B^-=\emptyset$ (see Fig.1(b)). In this case, the measures of sets $A^\pm$ are as follows:
\begin{equation}
\mu(A^-)=\mu_c^-:=\frac{LC_{max}-\tau\overline{C}\overline{v}}{v_{max}(C_{max}-C_{min})}\qquad \mu(A^+)=\mu_c^+:=\tau-\mu_c^-.
\end{equation}
    
\end{itemize}

\begin{figure}[H]
\centering
\subfloat[Case $(i)$]{
\scalebox{0.6}{\includegraphics{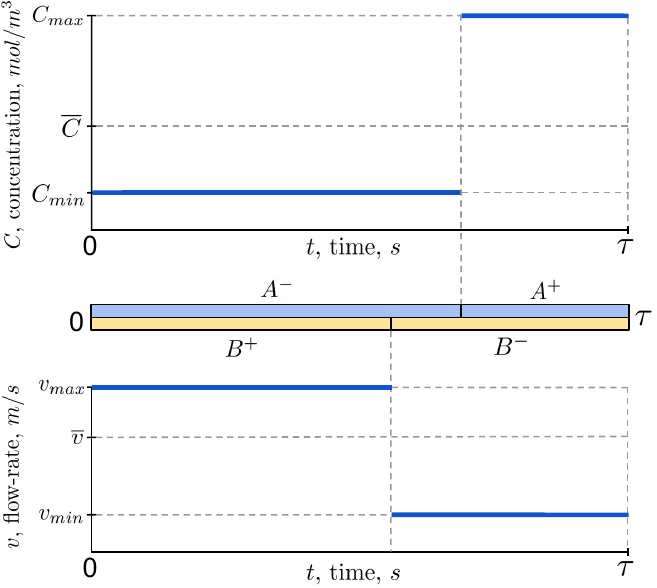}}
}
\hspace{0mm}
\subfloat[Case $(ii)$]{
\scalebox{0.6}{\includegraphics{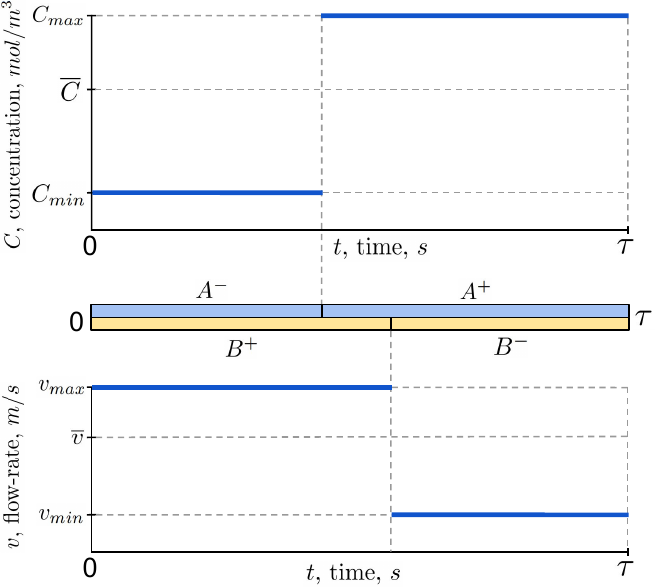}}
}
\caption{Illustration of possible cases described in Definition~\ref{B_class}.
}
\label{fig0}
\end{figure}

Note that all admissible controls from the class $\mathcal{B}_\tau$ are equivalent in the sense of Problem~\ref{prob2}, namely, the cost functional $J$ takes the same value for all controls $(c,v)\in\mathcal{B}_\tau$.
\end{definition}

The result of solving the isoperimetric optimal control problem (Problem~\ref{prob2}) is presented in the following theorem.

\begin{theorem}\label{th2}
Let assumptions~\eqref{assump2} and \eqref{assump2*}  be satisfied.
\begin{itemize}
\item[1)]If $n=1$, then controls $C_{A_0}$ and $v$ have no impact on the value of the cost functional $J$, namely,
\begin{equation*}
J[C_{A_0},v]=\overline{C}\,\overline{v}e^{-k\tau}\qquad\forall\,(C_{A_0},v)\in\mathcal{V}_\tau.
\end{equation*}

\item[2)]
If $n>1$, then the class of controls $\mathcal{B}_\tau$ is the optimal strategy for Problem~\ref{prob2}.
\end{itemize}
\end{theorem}

To prove this theorem we need auxiliary lemmas.

\begin{lemma}\label{lemma_period}
If the functions $(C_{A_0},v)\in\mathcal{V}_\tau$, then the corresponding solution $C(x,t)$ of the problem~\eqref{eq2}, \eqref{bound2}  is also $\tau$-periodic with respect to $t$.
\end{lemma}

\begin{proof}
Since $v$ is a $\tau$-periodic function, we get for any $s\in[0,\tau]$:
\begin{equation*}
V(\tau)=\int_0^\tau v(\xi)d\xi=\int_{-s}^{\tau-s}v(\xi)d\xi=-V(-s)+V(\tau-s).
\end{equation*}
Under  assumption \eqref{assump2} and the monotonicity of $V$, for each $x\in[0,L]$ one can find an $s\in[0,\tau]$ such that $x=-V(-s)$. Using this fact and the $\tau$--periodicity of the functions $C_{A_0}$ and $v$, we get:
\begin{align*}
C_A(x,\tau)&=\left[C_{A_0}\left(V^{-1}(V(\tau)-x)\right)^{-(n-1)}+k(n-1)\left(\tau-V^{-1}(V(\tau)-x)\right)\right]^{-\frac1{n-1}}
\\& =\left[C_{A_0}\left(V^{-1}(V(\tau-s))\right)^{-(n-1)}+k(n-1)\left(\tau-V^{-1}(V(\tau-s))\right)\right]^{-\frac1{n-1}}
\\& =\left[C_{A_0}(\tau-s)^{-(n-1)}+k(n-1)s\right]^{-\frac1{n-1}}
\\& =\left[C_{A_0}\left(-V^{-1}(-x)\right)^{-(n-1)}+k(n-1)\left(-V^{-1}(-x)\right)\right]^{-\frac1{n-1}}
\\& =C_A(x,0).
\end{align*}
\end{proof}

\begin{lemma}\label{lemma_cost}
Let the control functions $(C_{A_0},v)\in\mathcal{V}_\tau$ then the cost functional \eqref{cost2} of Problem~\ref{prob2} can be rewritten as follows:
\begin{align}
& J[C_{A_0},v]=\frac1\tau\int_0^\tau C_{A_0}(t)e^{-k\tau}v(t)dt,\quad\text{if } n=1\label{cost2**},
\\&J[C_{A_0},v]=\frac1\tau\int_0^\tau \Psi(C_{A_0}(t))v(t)dt,\quad\text{if } n\neq1,\label{cost2*}
\end{align}
where $\Psi(\xi):=\left(\xi^{-(n-1)}+k(n-1)\tau\right)^{-\frac1{n-1}}$ is an increasing concave function in the case $n>1$.
\end{lemma}

\begin{proof}
Using assumption \eqref{assump2}, the $\tau$--periodicity of the function $v$ and the definition of function $V(t)$, we get
\begin{equation*}
V^{-1}(V(t)-L)=V^{-1}(V(t)-V(\tau))=V^{-1}(V(t-\tau))=t-\tau,
\end{equation*}
which allows us to rewrite the cost functional \eqref{cost2} in the case $n\neq1$ as follows:
\begin{equation*}
\begin{aligned}
J[C_{A_0},v]&=\frac1\tau\int_0^\tau \left[C_{A_0}\left(t-\tau\right)^{-(n-1)}+k(n-1)\left(t-(t-\tau))\right)\right]^{-\frac1{n-1}}v(t)dt
\\&=\frac1\tau\int_0^\tau \left[C_{A_0}(t)^{-(n-1)}+k(n-1)\tau\right]^{-\frac1{n-1}}v(t)dt,
\end{aligned}
\end{equation*}
provided that the function $C_{A_0}$ is $\tau$--periodic. Calculating the derivatives of the function $\Psi$ as it was done in \eqref{Phi'}, \eqref{Phi''}, we conclude that, if $n>1$, $\Psi$ is increasing and concave.

Similarly, under the conditions of Lemma~\ref{lemma_cost}, we obtain \eqref{cost2**} in the case $n=1$. \end{proof}

For further analysis, we introduce an auxiliary class of control functions.

\begin{definition}\label{def3.1}
A vector-function $(c,v):{\mathbb R}\to [C_{min},C_{max}]\times[v_{min},v_{max}]$ belongs to the class $\mathcal{A}_{\,\widetilde{C},\,\nu}$ for given constants $\widetilde{C}\in [C_{min},C_{max}]$, $\nu\in[0,\tau]$ if  $(c,v)\in {\cal V}_\tau$ and there exist Lebesgue-measurable sets $A^+\subset [0,\tau)$,  $A^-\subset [0,\tau)$ such that:
\begin{itemize}
    \item[1)] $\essinf_{t\in\mathbb{A}^+}c(t)\geqslant\widetilde{C}$;

    \item[2)] $\esssup_{t\in\mathbb{A}^-}c(t)\leqslant\widetilde{C}$;

    \item[3)] $\mu(A^+\cap A^-)=0$, $\mu(A^+\cup A^-)=\tau$; 
    
    \item[4)] $\mu(A^+)=\nu$.
    
\end{itemize}
\end{definition}

\begin{lemma}\label{lem3.4}
Let $n>1$ and let $(C_\mathcal{A},v)\in \mathcal{A}_{\,\widetilde{C},\,\mu_c^+}$ for some $\widetilde{C}\in(C_{min},C_{max})$, with $\mu_c^+$ from Definition~\ref{B_class}. Then there exist control functions $(C_b,v_b)\in {\cal B}_\tau$ such that
$$
J [C_b,v_b] \leqslant J [C_\mathcal{A},v]
$$
where the cost $J$ is defined in Problem~\ref{prob2}.
\end{lemma}

\begin{proof}
Consider an arbitrary vector-function $(C_\mathcal{A}(t),v(t))$ from the class $\mathcal{A}_{\,\widetilde{C},\,\mu_c^+}$ with a fixed $\widetilde{C}\in(C_{min},C_{max})$ and $\mu_c^+$ defined in Definition~\ref{B_class}.
Define now the corresponding bang-bang control $c_b:[0,\tau)\to [C_{min]},C_{max}]$:
\begin{equation}\label{bang4}
c_{b}(t)=
\begin{cases}
C_{min}, &\text{ if }t\in A^-,
\\ C_{max}, &\text{ if }t\in A^+,
\end{cases}
\end{equation}
where  the sets $A^+$, $A^-$ correspond to the class $\mathcal{A}_{\,\widetilde{C},\,\mu_c^+}$. 
Set $C_b=c_b^\tau$, the $\tau$-periodic extension of $c_b$. Due to Definition~\ref{def3.1}, $\mu(A^+)=\mu_c^+$ and $\mu(A^-)=\mu_c^-$, so a control function $v_b$ can be constructed, so that $(C_b,v_b)\in\mathcal{B}_\tau$.

Now using the isoperimetric conditions \eqref{isoperim2}, \eqref{assump2} and the  concavity property of the  function $\Psi$:
\begin{equation*}
\Psi(x)-\Psi(y)\leqslant  \Psi'(y)(x-y)\quad\forall\,x,y\in[C_{min},C_{max}],
\end{equation*}
we investigate the difference of costs $J[C_b,v_b]-J[C_\mathcal{A},v]$: 
\begin{equation*}
\begin{aligned}
&J[C_b,v_b]-J[C_\mathcal{A},v]=\frac1\tau\int_0^\tau \left(\Psi(C_b(t))v_b(t)-\Psi(C_\mathcal{A}(t))v(t)\right)dt 
\\&=\frac1\tau\int_0^\tau \Psi(C_b(t))(v_b(t)-v(t))dt
+\frac1\tau\int_0^\tau \left(\Psi(C_b(t))-\Psi(C_\mathcal{A}(t))\right)v(t)dt
\\&=:I_1+I_2
\end{aligned}
\end{equation*}
Now we estimate the integral $I_1$ in cases $(i)$ and $(ii)$ (see Definition~\ref{B_class}) separately. In case $(i)$, due to assumption \eqref{assump2}, we have:
\begin{equation*}
\begin{aligned}
I_1=&\frac1\tau\int_0^\tau \Psi(C_b(t))(v_b(t)-v(t))dt
=\frac1\tau\int_{A^+} \Psi(C_{max})(v_{min}-v(t))dt
\\&+\frac1\tau\int_{A^-\cap B^-} \Psi(C_{min})(v_{min}-v(t))dt+ \frac1\tau\int_{B^+} \Psi(C_{min})(v_{max}-v(t))dt
\\\leqslant & \frac1\tau\int_0^\tau \Psi(C_{min})(v_b(t)-v(t))dt=0
\end{aligned}
\end{equation*}
In case $(ii)$, by similar logic, we have:
\begin{equation*}
\begin{aligned}
I_1=&\frac1\tau\int_0^\tau \Psi(C_b(t))(v_b(t)-v(t))dt
=\frac1\tau\int_{A^+} \Psi(C_{max})(v_{min}-v(t))dt
\\&+\frac1\tau\int_{A^+\cap B^+} \Psi(C_{max})(v_{max}-v(t))dt+ \frac1\tau\int_{B^+} \Psi(C_{min})(v_{max}-v(t))dt
\\\leqslant &\frac1\tau\int_0^\tau \Psi(C_{max})(v_b(t)-v(t))dt=0
\end{aligned}
\end{equation*}
The integral $I_2$ can be estimated similarly as in the proof of Lemma \ref{lem1}:
\begin{equation*}
\begin{aligned}
I_2&=\frac1\tau\int_0^\tau \left(\Psi(C_b(t))-\Psi(C_\mathcal{A}(t))\right)v(t)dt
\leqslant\frac 1\tau\int_0^\tau\Psi'(C_\mathcal{A}(t))\left(C_b(t)-C_\mathcal{A}(t)\right)v(t)dt
\\&\leqslant\frac 1\tau\Psi'(\widetilde{C})\int_0^\tau\left(C_{b}(t)-C_\mathcal{A}(t)\right)v(t)dt=0.
\end{aligned}
\end{equation*}
The obtained estimates prove Lemma~\ref{lem3.4}.
\end{proof}

\begin{lemma}\label{lem3.5}
For any vector-function $(u(\cdot),v(\cdot))\in\mathcal{V}_{\,\tau}$, there exists a constant $\widetilde{C}\in(C_{min},C_{max})$ such that   $(u(\cdot),v(\cdot))\in\mathcal{A}_{\,\widetilde{C},\mu_c^+}$, where $\mu_c^+$ is from Definition~\ref{B_class}.
\end{lemma}

\begin{proof}
Assume that there exists a vector-function $(u,v)\in\mathcal{U}_{\,\tau}$ which does not belong to any class $\mathcal{A}_{\,\widetilde{C},\mu_c^+}$.  Due to Definition~\ref{def3.1}, this means that, for any $\widetilde{C}\in(C_{min},C_{max})$, either
\begin{equation}\label{mes++}
    \mu(\{t\in[0,\tau):u(t)>\widetilde{C}\})>\mu_c^+,
\end{equation}
or
\begin{equation}\label{mes--}
    \mu(\{t\in[0,\tau):u(t)<\widetilde{C}\})>\mu_c^-.
\end{equation}

By the same logic as in Lemma~\ref{lem2}, we obtain that  $u(t)=C_{max}$ for $\mu$-a.a. $t\in A^+_\delta$ in the case \eqref{mes++}. Now we check if the isoperimetric constraint \eqref{isoperim2} holds. We will consider two cases $(i)$ and $(ii)$ separately. In case $(i)$, we have:
\begin{align*}
\frac 1\tau\int_0^\tau u(t)v(t)\,dt&=\frac 1\tau\int_{A^+_\delta} C_{max}v(t)\,dt+\frac 1\tau\int_{[0,\tau)\setminus A^+_\delta} u(t)v(t)\,dt
\\&\geqslant \frac 1\tau\int_{A^+_\delta} C_{max}v(t)\,dt+\frac 1\tau\int_{[0,\tau)\setminus A^+_\delta} C_{min}v(t)\,dt
\\&\geqslant \frac 1\tau\int_{A^+_\delta} (C_{max}-C_{min})v_{min}\,dt+\frac 1\tau C_{min}L
\\&>\frac 1\tau\mu_c^+ (C_{max}-C_{min})v_{min}+\frac 1\tau C_{min}L
=\overline{C}\overline{v}.
\end{align*}
In case $(ii)$, we have:
\begin{align*}
\frac 1\tau\int_0^\tau u(t)v(t)\,dt&=\frac 1\tau\int_{A^+_\delta} C_{max}v(t)\,dt+\frac 1\tau\int_{[0,\tau)\setminus A^+_\delta} u(t)v(t)\,dt
\\&\geqslant \frac 1\tau\int_{A^+_\delta} C_{max}v(t)\,dt+\frac 1\tau\int_{[0,\tau)\setminus A^+_\delta} C_{min}v(t)\,dt
\\&\geqslant \frac 1\tau C_{max}L- \frac 1\tau\int_{[0,\tau)\setminus A^+_\delta} (C_{max}-C_{min})v_{max}\,dt
\\&>\frac 1\tau C_{max}L-\frac 1\tau\mu_c^-(C_{max}-C_{min})v_{max}
=\overline{C}\,\overline{v}.
\end{align*}

Thus we get that the isoperimetric constraint \eqref{isoperim2} is violated in both cases, so $(u,v)\notin\mathcal{V}_{\,\tau}$, which contradicts our assumption.

Using the same arguments, one can prove that $\frac 1\tau\int_0^\tau u(t)v(t)\,dt<\overline{C}\,\overline{v}$ in the case \eqref{mes--}.
\end{proof}

\begin{proof}[Proof of Theorem \ref{th2}.]

To solve Problem~\ref{prob2} under assumptions \eqref{assump2}, \eqref{assump2*}
in the case $n=1$, we minimize the cost functional~\eqref{cost2**} because of Lemma~\ref{lemma_cost}. Computing  the cost value directly, we obtain
\begin{equation*}
J[C_{A_0},v]=\frac1\tau\int_0^\tau C_{A_0}(t)e^{-k\tau}v(t)dt=\overline{C}\,\overline{v}e^{-k\tau},
\end{equation*}
which proves the first assertion of the theorem.

For the case $n>1$, due to Lemmas \ref{lemma_period}, \ref{lemma_cost}, \ref{lem3.4}, \ref{lem3.5}, we conclude that the control functions from $\mathcal{B}_\tau$ have the best performance in terms of Problem~\ref{prob2}, so the bang-bang strategy is optimal.  

\end{proof}

\begin{remark}\label{freq2}
Considering the proposed bang-bang control strategy, we note that the number of switchings of $v(t)$ and $C_{A_0}(t)$ is not crucial
due to the fact that the integrand in~\eqref{cost2*} does not depend on $t$ explicitly.
So, from the theoretical viewpoint, the switching frequency can be chosen in an arbitrary way that preserves the established measures of $A^\pm$ and $B^\pm$.
However, it may not be desirable to switch too often from a practical point of view. Thus, the simplest optimal control strategy is parameterized by the sets $A^\pm$ and $B^\pm$ in the form of intervals, so that each control $C(t)$ and $v(t)$ has one switching per period $\tau$ as depicted in Fig.\,\ref{fig0}.
\end{remark}

\section{Case study}\label{case_study}

\subsection{Comparison of different control strategies}

In this section, we consider the system~\eqref{eq1}, \eqref{bound1} under the following choice of model parameters (cf.~\cite{MS-MP2008}):
\begin{equation}\label{values}
\begin{aligned}
\overline{C}&=1\ mol\,m^{-3};
\qquad &&n=2;
\qquad &&&k=0.001\ s^{-1}mol^{-1};
\\C_{max}&=1.5\ mol\,m^{-3};
\qquad && L=1\ m;
\  &&&v=0.01\ m\,s^{-1};
\\C_{min}&=0.5\ mol\,m^{-3};
\qquad &&\tau=100\ s;
\end{aligned}
\end{equation}
We define the sinusoidal function
\begin{equation*}
C_{sin}(t)=\overline{C}+(\overline{C}-C_{min})\sin\left(\frac{2\pi}{\tau}t\right)=
1+0.5\sin\frac{\pi t}{50}
\end{equation*}
and compare it with the bang-bang control function $C_b(t)$:
\begin{equation*}
C_b(t)=
\begin{cases}
1.5, &\text{ if }t\in [0,50),
\\ 0.5, &\text{ if }t\in [50,100),
\end{cases} \text{ and } \; C_b(t+100) = C_b(t) \; \text {for all} \; t\in {\mathbb R}.
\end{equation*}
Now we compute directly the cost functional for both functions:
\begin{equation*}
\begin{aligned}
    J[C_{sin}]&=\frac1\tau\int_0^\tau \Phi(C_{sin}(t))\,v\, dt= 10^{-4}\int_0^{100}\frac{10\,C_1(t)}{10+C_1(t)}dt
    = 10^{-4}\int_0^{100}\frac{10+5\sin\left(\frac{\pi t}{50}\right)}{11+0.5\sin\left(\frac{\pi t}{50}\right)}dt
    \\&\approx 8.9968\cdot 10^{-3} \left(\frac{mol}{m^2\,s}\right),
\end{aligned}
\end{equation*}
\begin{equation*}
\begin{aligned}
    J[C_b]&=\frac v\tau\int_0^\tau \Phi(C_b(t)) dt
    = 10^{-4}\int_0^{50}\frac{10\,C_{max}}{10+C_{max}}dt+ 10^{-4}\int_{50}^{100}\frac{10\,C_{min}}{10+C_{min}}dt\\&\approx 8.9027\cdot 10^{-3} \left(\frac{mol}{m^2\,s}\right).
\end{aligned}
\end{equation*}
So, the difference between these costs is approximately $9.416\cdot 10^{-5}\,{mol}\,{m^{-2}s^{-1}}$, which illustrates that the bang-bang strategy has more $A$ consumed and thus the performance is about $1.05\,\%$ better than for the sinusoidal one.

Calculating  for comparison  the cost function for the conventional steady-state operation
\begin{equation*}
    J[\overline{C}]=\frac1\tau\int_0^\tau \Phi(\overline{C})\,v\, dt= 10^{-2}\int_0^{100}\frac{10\,\overline{C}}{10+\overline{C}}\,dt\approx 9.0909\cdot 10^{-3} \left(\frac{mol}{m^2\,s}\right),
\end{equation*}
we  see that the bang-bang strategy is $2.07\,\%$ better than the steady-state control.

Now we also evaluate the performance of the case with controlled flow-rate and compare it with the above results. Consider the following control function which allows flow-rate modulations also deviating by $50\,\%$ from the steady state values:
\begin{equation*}
v(t)=
\begin{cases}
0.005, &\text{ if }t\in [0,50)\cup\{100\},
\\ 0.015, &\text{ if }t\in [50,100),
\end{cases}
\text{ and } \; v(t+100) = v(t) \; \text {for all} \; t\in {\mathbb R}.
\end{equation*}
and calculate the cost functional~\eqref{cost2*} for the parameters \eqref{values}:
\begin{equation*}
\begin{aligned}
    J[C_b,v]&=\frac1\tau\int_0^\tau \left[C_b(t)^{-(n-1)}+k(n-1)\tau\right]^{-\frac1{n-1}}v(t)dt\\&=
    10^{-2}\int_0^{50}\frac{10\,C_{max}}{10+C_{max}}\,v_{min}\,dt+  10^{-2}\int_{50}^{100}\frac{10\,C_{min}}{10+C_{min}}\,v_{max}\,dt \approx 6.8323\cdot 10^{-3} \left(\frac{mol}{m^2\,s}\right).
\end{aligned}
\end{equation*}
Thus, the predicted performance of the two-input control strategy for this scenario is $23.26\,\%$ better than the single-input bang-bang strategy, and is $24.84\,\%$ better than the steady-state.

\subsection{The impact of ``forcing parameters''}

In this subsection, we investigate the role of control design parameters that could impact the performance of the reaction model. Usually, authors investigate the frequency of switching, the phase shift with  two controls, and the amplitude as ``forcing parameters'' (see, e.g., \cite{F21}).

Once the period of operation $\tau$ is fixed, the \textbf{frequency} of the switching for the bang-bang strategy does not impact the performance (see Remarks~\ref{freq1}, \ref{freq2}). So, the number and the frequency of switchings can be chosen arbitrarily, provided that the measure of appropriate sets in Theorems~\ref{th1} and~\ref{th2} is preserved.

It has been noted in the proof of Theorem~\ref{th2} that the main principle of optimizing the \textbf{phase shift} (as the difference between switching times of the two controls) is that the minimum concentration should be at the same time when the maximum flow-rate is applied. So, if $\overline{C}-C_{min}=C_{max}-\overline{C}$ and $\overline{v}-v_{min}=v_{max}-\overline{v}$, then the switching point(-s) should be the same for both controls, and the values should be opposite. 

Now we study the \textbf{amplitude} impact. We  call  $\alpha:=\frac{\overline{C}-C_{min}}{\overline{C}}$  the \textit{concentration amplitude} under the assumption  $\overline{C}-C_{min}=C_{max}-\overline{C}$. Similarly, we define the \textit{flow-rate amplitude} as the value $\beta:=\frac{\overline{v}-v_{min}}{\overline{v}}$ under the assumption   $\overline{v}-v_{min}=v_{max}-\overline{v}$. It is obvious that these amplitude values are considered in the range $(0,1)$.

For the single-input case from Section~\ref{PFR}, we consider the cost function for bang-bang controls as a function of the concentration amplitude $\alpha$. If the parameters are given by~\eqref{values} with varying concentration amplitude $\alpha$, then the cost functional~\eqref{cost1*} takes the form
\begin{equation}\label{cost_values_1}
\begin{aligned}
    J[C_{b}]&=J_\alpha[C_{b}]=\frac v\tau\int_0^\tau \Phi(C_b(t))\,dt
    = 5\cdot10^{-3}\left(\Phi\left(\overline{C}(1+\alpha)\right)+ \Phi\left(\overline{C}(1-\alpha)\right)\right)
    \\&= 5\cdot10^{-3}\left(\frac{10(1+\alpha)}{11+\alpha}+ \frac{10(1-\alpha)}{11-\alpha}\right)=
    \frac{\alpha^2-11}{10\alpha^2-1210}.
\end{aligned}
\end{equation}
The amplitude dependence is illustrated in Fig.2(a). One can see that the larger the amplitude, the better the performance is. Now we investigate the potential for improvement. For this purpose, we define the following function which shows the percentage of improvement in comparison with the steady-state:
\begin{equation}\label{perc_values_1}
    P_{C_{b},\overline{C}}(\alpha)=100\left(1-\frac{J_\alpha[C_b]}{J[\overline{C}]}\right)= 
    \frac{1000\alpha^2}{121-\alpha^2}.
\end{equation}

The graph of the function $P_{C_{b},\overline{C}}(\alpha)$ is shown in Fig.2(b). As one can see, the performance can be improved up to $8.26\,\%$ in comparison with the steady-state. But it should be taken into account that, in the limiting case $\alpha=1$, the concentration is zero for half of the time period (and thus there is no chemical reaction at all).

\begin{figure}[H]
\centering
\subfloat[$J$, Cost functional]{
\scalebox{0.55}{\input{cost1.tikz}}
}
\hspace{0mm}
\subfloat[$P$, Percentage function]{
\scalebox{0.55}{\input{perc1.tikz}}
}
\caption{The graphs of the cost functional $J$ from \eqref{cost_values_1} and the percentage function $P$ from \eqref{perc_values_1} as functions of the concentration amplitude $\alpha$ under the choice of parameters~\eqref{values}.
}
\label{fig1}
\end{figure}
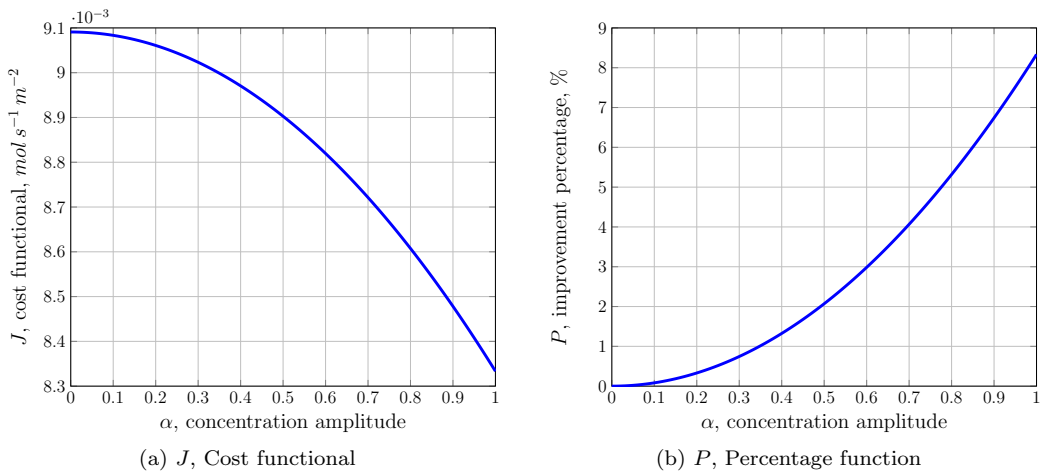

Considering the case with two control functions from Section~\ref{PFR2}, we can investigate the same dependencies. Namely, the cost functional has the following form:
\begin{equation}\label{cost_values_2}
\begin{aligned}
    J[C_{b},v_b]&=J_{\alpha,\beta}[C_{b},v_b]=\frac1\tau\int_0^\tau \Psi(C_b(t))\,v_b(t)\,dt =\frac{\overline{v}}{2}\left(\Psi\left(\overline{C}(1+\alpha)\right)(1-\beta)+ \Psi\left(\overline{C}(1-\alpha)\right)(1+\beta)\right)
    \\&=5\cdot10^{-3}\left(\frac{10(1+\alpha)(1-\beta)}{11+\alpha}+ \frac{10(1-\alpha)(1+\beta)}{11-\alpha}\right)
    =\frac{\alpha^2+10\alpha\beta-11}{10\alpha^2-1210}.
\end{aligned}
\end{equation}
Thus, the percentage function takes the form:
\begin{equation}\label{perc_values_2}
    P_{C_{b},\overline{C}}(\alpha,\beta)=100\left(1-\frac{J_{\alpha,\beta}[C_b,v_b]}{J[\overline{C},\overline{v}]}\right)= 
    \frac{1000\,\alpha\,(\alpha+11\beta)}{121-\alpha^2}.
\end{equation}

The graphs of the cost  and the percentage function are presented in Fig.3(a) and Fig.3(b), respectively. In this case, the performance could be improved theoretically up to $100\,\%$ (in the case of maximum concentration and flow-rate amplitudes) in comparison with the steady-state.

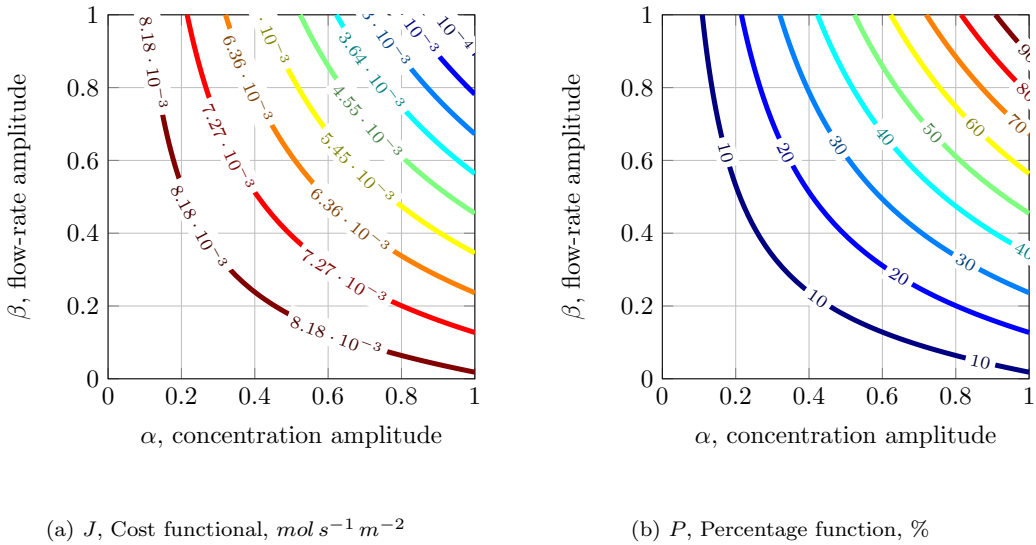
\begin{figure}[H]
\centering
\subfloat[$J$, Cost functional, $mol\,s^{-1}\,m^{-2}$]{
\scalebox{0.95}{\input{cost2.tikz}}
}
\hspace{0mm}
\subfloat[$P$, Percentage function, $\%$]{
\scalebox{0.95}{\input{perc2.tikz}}
}
\caption{The cost $J$ from \eqref{cost_values_2} and the percentage function $P$ from \eqref{perc_values_2} 
depending on the concentration amplitude $\alpha$ and flow-rate amplitude $\beta$.
The levels with specified values are marked with color lines.
}
\label{fig2}
\end{figure}

\begin{remark}
In the limiting cases $\alpha=1$ or $\beta=1$, there are  situations when the flow-rate is stopped ($v=0$) or the inlet concentration vanishes ($C_{A_0}=0$), so the reaction is stopped. In these cases, the investigated mathematical model cannot be used to describe the ongoing process. Such limiting cases should be treated separately as our analysis is not  directly applicable.
\end{remark}

\section{Conclusion and future work}\label{conc}

The optimal control problem for an isothermal PFR model with a single periodic input has been completely  solved in Section~\ref{PFR}, and the bang-bang optimal control strategies for periodic chemical reactions in an isothermal PFR with a controlled flow-rate have been proposed in Section~\ref{PFR2}. 

An open problem is to study the relevant isoperimetric optimal control problem for a more realistic {\em non-isothermal} PFR model (see, e.g.,~\cite{TCR87,WWL2011}).

The optimization of periodic reactions in a Dispersed Flow Tubular Reactor (DFTR) is also a challenging task in this direction. Namely, a  parabolic DFTR model~\cite[p.~394]{SH13} could be considered in 
future.


\section*{Acknowledgments}%
\addcontentsline{toc}{section}{Acknowledgments}

The second author was supported by the German Research Foundation (DFG) under Grant ZU 359/2-1.


\addcontentsline{toc}{section}{References}
\bibliographystyle{plainurl}

\end{document}

%% file: cost1.tikz
%
%
\begin{tikzpicture}

\begin{axis}
[width=4in,
height=3.4in,
at={(0.758in,0.481in)},
scale only axis,
xmin=0,
xmax=1,
xlabel style={font=\color{white!15!black}},
xlabel={$\alpha$, concentration amplitude},
ymin=0.0083,
ymax=0.0091,
ylabel style={font=\color{white!15!black}},
ylabel={$J$, cost functional, $mol\,s^{-1}\,m^{-2}$},
axis background/.style={fill=white},
label style = {font=\Large},
tick label style = {font=\large},
xmajorgrids,
ymajorgrids
]
\addplot [color=blue, line width=2.0pt, forget plot]
  table[row sep=crcr]{%
0	0.00909090909090909\\
0.01	0.00909083395936691\\
0.02	0.00909060856399525\\
0.03	0.00909023290255878\\
0.04	0.00908970697133185\\
0.05	0.00908903076509845\\
0.06	0.00908820427715205\\
0.07	0.00908722749929543\\
0.08	0.00908610042184049\\
0.09	0.00908482303360803\\
0.1	0.00908339532192743\\
0.11	0.00908181727263636\\
0.12	0.00908008887008041\\
0.13	0.00907821009711274\\
0.14	0.00907618093509362\\
0.15	0.00907400136388998\\
0.16	0.00907167136187491\\
0.17	0.00906919090592712\\
0.18	0.00906655997143037\\
0.19	0.00906377853227285\\
0.2	0.00906084656084656\\
0.21	0.00905776402804659\\
0.22	0.0090545309032704\\
0.23	0.0090511471544171\\
0.24	0.0090476127478866\\
0.25	0.00904392764857881\\
0.26	0.00904009181989276\\
0.27	0.0090361052237257\\
0.28	0.00903196782047211\\
0.29	0.00902767956902277\\
0.3	0.00902324042676371\\
0.31	0.00901865034957516\\
0.32	0.00901390929183044\\
0.33	0.00900901720639485\\
0.34	0.00900397404462445\\
0.35	0.00899877975636491\\
0.36	0.00899343428995023\\
0.37	0.00898793759220142\\
0.38	0.00898228960842526\\
0.39	0.00897649028241285\\
0.4	0.00897053955643826\\
0.41	0.00896443737125709\\
0.42	0.00895818366610497\\
0.43	0.00895177837869604\\
0.44	0.00894522144522144\\
0.45	0.00893851280034769\\
0.46	0.00893165237721503\\
0.47	0.00892464010743581\\
0.48	0.00891747592109273\\
0.49	0.00891015974673712\\
0.5	0.00890269151138716\\
0.51	0.00889507114052604\\
0.52	0.00888729855810008\\
0.53	0.00887937368651688\\
0.54	0.00887129644664332\\
0.55	0.0088630667578036\\
0.56	0.00885468453777725\\
0.57	0.00884614970279701\\
0.58	0.0088374621675468\\
0.59	0.0088286218451595\\
0.6	0.00881962864721485\\
0.61	0.00881048248373718\\
0.62	0.00880118326319315\\
0.63	0.0087917308924895\\
0.64	0.00878212527697064\\
0.65	0.00877236632041633\\
0.66	0.00876245392503923\\
0.67	0.00875238799148245\\
0.68	0.00874216841881703\\
0.69	0.00873179510453943\\
0.7	0.00872126794456892\\
0.71	0.00871058683324495\\
0.72	0.00869975166332452\\
0.73	0.00868876232597946\\
0.74	0.00867761871079364\\
0.75	0.00866632070576025\\
0.76	0.00865486819727891\\
0.77	0.00864326107015284\\
0.78	0.00863149920758591\\
0.79	0.00861958249117971\\
0.8	0.00860751080093054\\
0.81	0.00859528401522636\\
0.82	0.00858290201084373\\
0.83	0.00857036466294465\\
0.84	0.00855767184507342\\
0.85	0.00854482342915342\\
0.86	0.00853181928548383\\
0.87	0.00851865928273639\\
0.88	0.00850534328795198\\
0.89	0.00849187116653731\\
0.9	0.00847824278226142\\
0.91	0.00846445799725227\\
0.92	0.00845051667199318\\
0.93	0.00843641866531929\\
0.94	0.00842216383441395\\
0.95	0.00840775203480506\\
0.96	0.00839318312036136\\
0.97	0.00837845694328876\\
0.98	0.00836357335412647\\
0.99	0.00834853220174321\\
1	0.00833333333333333\\
};
\end{axis}

\begin{axis}[%
width=2in,
height=2in,
at={(0in,0in)},
scale only axis,
xmin=0,
xmax=1,
ymin=0,
ymax=1,
axis line style={draw=none},
ticks=none,
axis x line*=bottom,
axis y line*=left
]
\end{axis}
\end{tikzpicture}%

%% file: perc1.tikz
%
%
\begin{tikzpicture}

\begin{axis}[%
width=4in,
height=3.4in,
at={(0.758in,0.481in)},
scale only axis,
xmin=0,
xmax=1,
xlabel style={font=\color{white!15!black}},
xlabel={$\alpha$, concentration amplitude},
ymin=0,
ymax=9,
ylabel style={font=\color{white!15!black}},
ylabel={$P$, improvement percentage, $\%$},
axis background/.style={fill=white},
label style = {font=\Large},
tick label style = {font=\large},
xmajorgrids,
ymajorgrids
]
\addplot [color=blue, line width=2.0pt, forget plot]
  table[row sep=crcr]{%
0	0\\
0.01	0.000826446964005755\\
0.02	0.00330579605221835\\
0.03	0.00743807185342701\\
0.04	0.0132233153496245\\
0.05	0.0206615839170231\\
0.06	0.0297529513274775\\
0.07	0.0404975077503139\\
0.08	0.0528953597545655\\
0.09	0.0669466303116159\\
0.1	0.0826514587982478\\
0.11	0.1000100010001\\
0.12	0.119022429115531\\
0.13	0.13968893175989\\
0.14	0.162009713970197\\
0.15	0.185984997210225\\
0.16	0.211615019376\\
0.17	0.238900034801701\\
0.18	0.267840314265969\\
0.19	0.298436144998632\\
0.2	0.330687830687831\\
0.21	0.364595691487559\\
0.22	0.40016006402561\\
0.23	0.43738130141194\\
0.24	0.47625977324743\\
0.25	0.516795865633075\\
0.26	0.558989981179568\\
0.27	0.602842539017309\\
0.28	0.648353974806817\\
0.29	0.695524740749562\\
0.3	0.744355305599206\\
0.31	0.794846154673257\\
0.32	0.846997789865142\\
0.33	0.900810729656691\\
0.34	0.956285509131038\\
0.35	1.01342267998594\\
0.36	1.0722228105475\\
0.37	1.13268648578433\\
0.38	1.19481430732213\\
0.39	1.25860689345864\\
0.4	1.32406487917908\\
0.41	1.39118891617197\\
0.42	1.45997967284537\\
0.43	1.53043783434355\\
0.44	1.6025641025641\\
0.45	1.67635919617542\\
0.46	1.75182385063466\\
0.47	1.82895881820613\\
0.48	1.90776486798002\\
0.49	1.98824278589167\\
0.5	2.0703933747412\\
0.51	2.15421745421356\\
0.52	2.23971586089907\\
0.53	2.32688944831431\\
0.54	2.41573908692353\\
0.55	2.5062656641604\\
0.56	2.59847008445028\\
0.57	2.69235326923284\\
0.58	2.78791615698521\\
0.59	2.88515970324545\\
0.6	2.9840848806366\\
0.61	3.08469267889104\\
0.62	3.18698410487532\\
0.63	3.29096018261554\\
0.64	3.39662195332298\\
0.65	3.50397047542037\\
0.66	3.61300682456845\\
0.67	3.72373209369305\\
0.68	3.83614739301264\\
0.69	3.95025385006625\\
0.7	4.06605260974193\\
0.71	4.18354483430557\\
0.72	4.30273170343023\\
0.73	4.42361441422596\\
0.74	4.54619418126995\\
0.75	4.67047223663726\\
0.76	4.79644982993197\\
0.77	4.92412822831876\\
0.78	5.05350871655498\\
0.79	5.18459259702316\\
0.8	5.31738118976404\\
0.81	5.45187583251\\
0.82	5.58807788071897\\
0.83	5.72598870760886\\
0.84	5.86560970419238\\
0.85	6.00694227931242\\
0.86	6.14998785967783\\
0.87	6.29474788989971\\
0.88	6.44122383252818\\
0.89	6.58941716808962\\
0.9	6.73932939512439\\
0.91	6.89096203022504\\
0.92	7.044316608075\\
0.93	7.19939468148776\\
0.94	7.35619782144653\\
0.95	7.5147276171444\\
0.96	7.674985676025\\
0.97	7.8369736238236\\
0.98	8.00069310460881\\
0.99	8.16614578082468\\
1	8.33333333333333\\
};
\end{axis}

\begin{axis}[%
width=2in,
height=2in,
at={(0in,0in)},
scale only axis,
xmin=0,
xmax=1,
ymin=0,
ymax=1,
axis line style={draw=none},
ticks=none,
axis x line*=bottom,
axis y line*=left
]
\end{axis}
\end{tikzpicture}%

%% file: cost2.tikz
%
%
\begin{tikzpicture}

\begin{axis}[%
width=2in,
height=2in,
at={(0.796in,0.708in)},
scale only axis,
point meta min=0.000909090909090909,
point meta max=0.00818181818181818,
colormap/jet,
xmin=0,
xmax=1,
xlabel style={font=\color{white!15!black}},
xlabel={$\alpha$, concentration amplitude},
ymin=0,
ymax=1,
ylabel style={font=\color{white!15!black}},
ylabel={$\beta$, flow-rate amplitude},
axis background/.style={fill=white},
label style = {font=\normalsize},
tick label style = {font=\normalsize},
xmajorgrids,
xminorgrids,
ymajorgrids,
yminorgrids
]
\addplot[contour prepared, contour prepared format=matlab, line width=2.0pt] table[row sep=crcr]
{%
0.000909090909090909	11\\
1	0.890909090909091\\
0.993554889012085	0.897959183673469\\
0.979591836734694	0.913556354359926\\
0.975346799219742	0.918367346938776\\
0.959183673469388	0.937081277385229\\
0.957741091619854	0.938775510204082\\
0.940703981569285	0.959183673469388\\
0.938775510204082	0.9615410986529\\
0.924211345073377	0.979591836734694\\
0.918367346938776	0.98699814471243\\
0.908246563931695	1\\
0.00181818181818182	21\\
1	0.781818181818182\\
0.985759942069542	0.795918367346939\\
0.979591836734694	0.802155225726654\\
0.96580037229958	0.816326530612245\\
0.959183673469388	0.823272411479098\\
0.946560620323227	0.836734693877551\\
0.938775510204082	0.845220617891425\\
0.928006694388141	0.857142857142857\\
0.918367346938776	0.868055246340961\\
0.910106749379132	0.877551020408163\\
0.897959183673469	0.891836734693878\\
0.892830915001095	0.897959183673469\\
0.877551020408163	0.916631142943435\\
0.876151140186564	0.918367346938776\\
0.860028902866534	0.938775510204082\\
0.857142857142857	0.942510822510823\\
0.844443903031278	0.959183673469388\\
0.836734693877551	0.969555183492466\\
0.829378902515349	0.979591836734694\\
0.816326530612245	0.997851576994434\\
0.814812010147287	1\\
0.00272727272727273	31\\
1	0.672727272727273\\
0.999142807237076	0.673469387755102\\
0.979591836734694	0.690754097093383\\
0.976123667098919	0.693877551020408\\
0.959183673469388	0.709463545572968\\
0.954036901273298	0.714285714285714\\
0.938775510204082	0.728900137129951\\
0.932832124350065	0.73469387755102\\
0.918367346938776	0.749112347969491\\
0.912462586215558	0.755102040816326\\
0.897959183673469	0.77015306122449\\
0.892884843621389	0.775510204081633\\
0.877551020408163	0.79208007938905\\
0.874058466703153	0.795918367346939\\
0.857142857142857	0.81495670995671\\
0.855945776186612	0.816326530612245\\
0.838503559542137	0.836734693877551\\
0.836734693877551	0.838852436761844\\
0.821699083083961	0.857142857142857\\
0.816326530612245	0.863843692022263\\
0.805507383032012	0.877551020408163\\
0.795918367346939	0.890014747157604\\
0.789899893271146	0.897959183673469\\
0.775510204081633	0.917458744263256\\
0.774849810745705	0.918367346938776\\
0.760310207274331	0.938775510204082\\
0.755102040816326	0.946278894850324\\
0.74627510567259	0.959183673469388\\
0.73469387755102	0.976589878375593\\
0.732725171847478	0.979591836734694\\
0.719618340842831	1\\
0.00363636363636364	41\\
1	0.563636363636364\\
0.989774827087873	0.571428571428571\\
0.979591836734694	0.579352968460111\\
0.963885548611758	0.591836734693878\\
0.959183673469388	0.595654679666838\\
0.93917061541492	0.612244897959184\\
0.938775510204082	0.612579656368476\\
0.918367346938776	0.630169449598021\\
0.915542624830204	0.63265306122449\\
0.897959183673469	0.648469387755102\\
0.892955008500296	0.653061224489796\\
0.877551020408163	0.667529015834664\\
0.871348478328982	0.673469387755102\\
0.857142857142857	0.687402597402597\\
0.850666352637952	0.693877551020408\\
0.836734693877551	0.708149690031223\\
0.830856302303764	0.714285714285714\\
0.816326530612245	0.729835807050093\\
0.811869933580195	0.73469387755102\\
0.795918367346939	0.75253318110461\\
0.793662417975057	0.755102040816326\\
0.776188912562954	0.775510204081633\\
0.775510204081633	0.776321648276536\\
0.759400319755883	0.795918367346939\\
0.755102040816326	0.80128967557539\\
0.743271955709224	0.816326530612245\\
0.73469387755102	0.827535559678417\\
0.727770592417459	0.836734693877551\\
0.714285714285714	0.855168831168831\\
0.712865210024122	0.857142857142857\\
0.698506593472433	0.877551020408163\\
0.693877551020408	0.884311906580814\\
0.684680169245949	0.897959183673469\\
0.673469387755102	0.915102040816327\\
0.671366985036112	0.918367346938776\\
0.658520725958792	0.938775510204082\\
0.653061224489796	0.947693645640074\\
0.646132122573246	0.959183673469388\\
0.634183073824009	0.979591836734694\\
0.63265306122449	0.982261056915435\\
0.622631479307864	1\\
0.00454545454545455	51\\
1	0.454545454545455\\
0.979591836734694	0.46795183982684\\
0.977441909478886	0.469387755102041\\
0.959183673469388	0.481845813760707\\
0.947816905997377	0.489795918367347\\
0.938775510204082	0.496259175607002\\
0.919732813096753	0.510204081632653\\
0.918367346938776	0.511226551226551\\
0.897959183673469	0.526785714285714\\
0.893050041375482	0.530612244897959\\
0.877551020408163	0.542977952280278\\
0.867700765101731	0.551020408163265\\
0.857142857142857	0.559848484848485\\
0.843602385182718	0.571428571428571\\
0.836734693877551	0.577446943300602\\
0.820671871393806	0.591836734693878\\
0.816326530612245	0.595827922077922\\
0.798833489043014	0.612244897959184\\
0.795918367346939	0.615051615051615\\
0.778018005534617	0.63265306122449\\
0.775510204081633	0.635184552289816\\
0.758161998244664	0.653061224489796\\
0.755102040816326	0.656300456300457\\
0.739207248932035	0.673469387755102\\
0.73469387755102	0.678481240981241\\
0.721100212321404	0.693877551020408\\
0.714285714285714	0.701818181818182\\
0.70379154847289	0.714285714285714\\
0.693877551020408	0.726413292589763\\
0.687235710182435	0.73469387755102\\
0.673469387755102	0.752380952380952\\
0.671390578004141	0.755102040816326\\
0.65620212079602	0.775510204081633\\
0.653061224489796	0.779849837662338\\
0.641637081699232	0.795918367346939\\
0.63265306122449	0.808965228320067\\
0.627672198552827	0.816326530612245\\
0.614267301702801	0.836734693877551\\
0.612244897959184	0.839891774891775\\
0.601378595124623	0.857142857142857\\
0.591836734693878	0.87281683833408\\
0.589002053325287	0.877551020408163\\
0.577088466156936	0.897959183673469\\
0.571428571428571	0.907954545454545\\
0.565625625158663	0.918367346938776\\
0.554589241469644	0.938775510204082\\
0.551020408163265	0.945550745550746\\
0.54394896222054	0.959183673469388\\
0.533697554360826	0.979591836734694\\
0.530612244897959	0.985889110889111\\
0.523800689106812	1\\
0.00545454545454545	62\\
1	0.345454545454546\\
0.997220256192968	0.346938775510204\\
0.979591836734694	0.356550711193568\\
0.960385237044188	0.36734693877551\\
0.959183673469388	0.368036947854577\\
0.938775510204082	0.379938694845527\\
0.92574839458151	0.387755102040816\\
0.918367346938776	0.392283652855081\\
0.897959183673469	0.405102040816327\\
0.893186010352028	0.408163265306122\\
0.877551020408163	0.418426888725892\\
0.862526787206612	0.428571428571429\\
0.857142857142857	0.432294372294372\\
0.836734693877551	0.446744196569981\\
0.833643708204024	0.448979591836735\\
0.816326530612245	0.461820037105751\\
0.80639019245086	0.469387755102041\\
0.795918367346939	0.47757004899862\\
0.780674088251268	0.489795918367347\\
0.775510204081633	0.494047456303096\\
0.756378252358699	0.510204081632653\\
0.755102040816326	0.511311237025523\\
0.73469387755102	0.529426922284065\\
0.733389143181811	0.530612244897959\\
0.714285714285714	0.548467532467532\\
0.711619858600667	0.551020408163265\\
0.693877551020408	0.568514678598712\\
0.690991476653237	0.571428571428571\\
0.673469387755102	0.589659863945578\\
0.671424991049051	0.591836734693878\\
0.653061224489796	0.612006029684601\\
0.65284859367528	0.612244897959184\\
0.635182765496083	0.63265306122449\\
0.63265306122449	0.635669399724699\\
0.618376661200732	0.653061224489796\\
0.612244897959184	0.66078169449598\\
0.602377381435418	0.673469387755102\\
0.591836734693878	0.687492802763739\\
0.587135287281123	0.693877551020408\\
0.572598787432558	0.714285714285714\\
0.571428571428571	0.715974025974026\\
0.558705432784024	0.73469387755102\\
0.551020408163265	0.746422043564901\\
0.545441117568252	0.755102040816326\\
0.532760030337407	0.775510204081633\\
0.530612244897959	0.779063793349508\\
0.52061133066425	0.795918367346939\\
0.510204081632653	0.814162523191095\\
0.508991769726657	0.816326530612245\\
0.497837039755259	0.836734693877551\\
0.489795918367347	0.852025355596784\\
0.487151048679594	0.857142857142857\\
0.476882649652582	0.877551020408163\\
0.469387755102041	0.893012825683633\\
0.467029804222143	0.897959183673469\\
0.457547386381263	0.918367346938776\\
0.448979591836735	0.937551020408164\\
0.44844144304132	0.938775510204082\\
0.439659069666481	0.959183673469388\\
0.431215111910851	0.979591836734694\\
0.428571428571429	0.986147186147186\\
0.423069175722237	1\\
0.00636363636363636	72\\
1	0.236363636363636\\
0.980164531677868	0.244897959183673\\
0.979591836734694	0.245149582560297\\
0.959183673469388	0.254228081948447\\
0.938775510204082	0.263618214084053\\
0.935166742973639	0.26530612244898\\
0.918367346938776	0.273340754483612\\
0.897959183673469	0.283418367346939\\
0.893396626319864	0.285714285714286\\
0.877551020408163	0.293875825171506\\
0.857142857142857	0.30474025974026\\
0.85459292682679	0.306122448979592\\
0.836734693877551	0.316041449839359\\
0.818499510985635	0.326530612244898\\
0.816326530612245	0.327812152133581\\
0.795918367346939	0.340088482945626\\
0.784881849491505	0.346938775510204\\
0.775510204081633	0.352910360316376\\
0.755102040816326	0.366322017750589\\
0.753574751059389	0.36734693877551\\
0.73469387755102	0.380372603586889\\
0.724330489688718	0.387755102040816\\
0.714285714285714	0.395116883116883\\
0.697028552389178	0.408163265306122\\
0.693877551020408	0.410616064607661\\
0.673469387755102	0.426938775510204\\
0.671479882190728	0.428571428571429\\
0.653061224489796	0.444162221706865\\
0.647534868538818	0.448979591836735\\
0.63265306122449	0.462373571129332\\
0.625078880164739	0.469387755102041\\
0.612244897959184	0.481671614100185\\
0.603988760937946	0.489795918367347\\
0.591836734693878	0.502168767193398\\
0.584154675956652	0.510204081632653\\
0.571428571428571	0.523993506493507\\
0.565478343563995	0.530612244897959\\
0.551020408163265	0.547293341579056\\
0.547871541623425	0.551020408163265\\
0.531250959763676	0.571428571428571\\
0.530612244897959	0.572238475809905\\
0.515525013228909	0.591836734693878\\
0.510204081632653	0.599026345083488\\
0.500648981905227	0.612244897959184\\
0.489795918367347	0.627887291280149\\
0.486564530060686	0.63265306122449\\
0.473197632904196	0.653061224489796\\
0.469387755102041	0.659091715737679\\
0.460500813163984	0.673469387755102\\
0.448979591836735	0.692959183673469\\
0.448448197375086	0.693877551020408\\
0.436956635811167	0.714285714285714\\
0.428571428571429	0.72987012987013\\
0.426028417862646	0.73469387755102\\
0.415596178762939	0.755102040816326\\
0.408163265306122	0.770281076066791\\
0.405652093077329	0.775510204081633\\
0.39614093730565	0.795918367346939\\
0.387755102040816	0.814744653842398\\
0.387063518567228	0.816326530612245\\
0.378357963225002	0.836734693877551\\
0.370036253407295	0.857142857142857\\
0.36734693877551	0.863936301793445\\
0.362050887232655	0.877551020408163\\
0.35439274712986	0.897959183673469\\
0.347052514177095	0.918367346938776\\
0.346938775510204	0.918690385245007\\
0.339982721022701	0.938775510204082\\
0.333195284580351	0.959183673469388\\
0.326674036762351	0.979591836734694\\
0.326530612244898	0.980049860853433\\
0.320380144636704	1\\
0.00727272727272727	82\\
1	0.127272727272727\\
0.979591836734694	0.133748453927025\\
0.959183673469388	0.140419216042316\\
0.951848234188443	0.142857142857143\\
0.938775510204082	0.147297733322578\\
0.918367346938776	0.154397856112142\\
0.897959183673469	0.161734693877551\\
0.893766619566548	0.163265306122449\\
0.877551020408163	0.16932476161712\\
0.857142857142857	0.177186147186147\\
0.840822942977849	0.183673469387755\\
0.836734693877551	0.185338703108738\\
0.816326530612245	0.19380426716141\\
0.795918367346939	0.202606916892631\\
0.792562082731039	0.204081632653061\\
0.775510204081633	0.211773264329655\\
0.755102040816326	0.221332798475656\\
0.74852720279158	0.224489795918367\\
0.73469387755102	0.231318284889714\\
0.714285714285714	0.241766233766234\\
0.708327033112342	0.244897959183673\\
0.693877551020408	0.25271745061661\\
0.673469387755102	0.26421768707483\\
0.671581272309506	0.26530612244898\\
0.653061224489796	0.276318413729128\\
0.637907131959287	0.285714285714286\\
0.63265306122449	0.289077742533964\\
0.612244897959184	0.302561533704391\\
0.607025605817456	0.306122448979592\\
0.591836734693878	0.316844731623057\\
0.578639456414104	0.326530612244898\\
0.571428571428571	0.332012987012987\\
0.552517806777557	0.346938775510204\\
0.551020408163265	0.348164639593211\\
0.530612244897959	0.365413158270301\\
0.528399829933993	0.36734693877551\\
0.510204081632653	0.383890166975881\\
0.506099515019483	0.387755102040816\\
0.489795918367347	0.403749226963513\\
0.485446232773535	0.408163265306122\\
0.469387755102041	0.425170605791724\\
0.466280057928946	0.428571428571429\\
0.448979591836735	0.448367346938776\\
0.448461209854011	0.448979591836735\\
0.43184336705587	0.469387755102041\\
0.428571428571429	0.473593073593074\\
0.416331977617587	0.489795918367347\\
0.408163265306122	0.501152133580705\\
0.401836264610305	0.510204081632653\\
0.388268715559101	0.530612244897959\\
0.387755102040816	0.531412947954301\\
0.375512262050755	0.551020408163265\\
0.36734693877551	0.564825809111523\\
0.363542821161879	0.571428571428571\\
0.352268473260463	0.591836734693878\\
0.346938775510204	0.601946960602422\\
0.341642959569235	0.612244897959184\\
0.331612337174164	0.63265306122449\\
0.326530612244898	0.643471706864564\\
0.322130694380215	0.653061224489796\\
0.313151193986413	0.673469387755102\\
0.306122448979592	0.690280766852195\\
0.304651800280709	0.693877551020408\\
0.296568136752082	0.714285714285714\\
0.288902942313773	0.73469387755102\\
0.285714285714286	0.743506493506494\\
0.281604544578788	0.755102040816326\\
0.274651796766055	0.775510204081633\\
0.268034500975934	0.795918367346939\\
0.26530612244898	0.804629656058227\\
0.261712627880715	0.816326530612245\\
0.25567080435079	0.836734693877551\\
0.249901924927102	0.857142857142857\\
0.244897959183673	0.875624613481756\\
0.244385721843202	0.877551020408163\\
0.23908798341788	0.897959183673469\\
0.234015250003436	0.918367346938776\\
0.229153485093298	0.938775510204082\\
0.224489795918367	0.959183673469388\\
0.224489795918367	0.959183673469388\\
0.219995006873326	0.979591836734694\\
0.215676801332778	1\\
0.00818181818181818	93\\
1	0.0181818181818182\\
0.988986088209275	0.0204081632653061\\
0.979591836734694	0.0223473252937538\\
0.959183673469388	0.0266103501361858\\
0.938775510204082	0.0309772525611034\\
0.918367346938776	0.0354549577406718\\
0.897959183673469	0.0400510204081632\\
0.894587004015467	0.0408163265306122\\
0.877551020408163	0.0447736980627344\\
0.857142857142857	0.0496320346320345\\
0.836734693877551	0.0546359563781166\\
0.816326530612245	0.0597963821892391\\
0.810754296816583	0.0612244897959184\\
0.795918367346939	0.0651253508396363\\
0.775510204081633	0.0706361683429352\\
0.755102040816326	0.076343579200722\\
0.736816900268051	0.0816326530612245\\
0.73469387755102	0.0822639661925377\\
0.714285714285714	0.0884155844155844\\
0.693877551020408	0.0948188366255592\\
0.673469387755102	0.101496598639456\\
0.67183166023166	0.102040816326531\\
0.653061224489796	0.108474605751391\\
0.63265306122449	0.115781913938596\\
0.614836822566901	0.122448979591837\\
0.612244897959184	0.123451453308596\\
0.591836734693878	0.131520696052716\\
0.571428571428571	0.140032467532467\\
0.56486416532993	0.142857142857143\\
0.551020408163265	0.149035937607366\\
0.530612244897959	0.158587840730698\\
0.521022395570714	0.163265306122449\\
0.510204081632653	0.168753988868275\\
0.489795918367347	0.179611162646877\\
0.482473187120357	0.183673469387755\\
0.469387755102041	0.191249495845769\\
0.448979591836735	0.203775510204081\\
0.448496685261211	0.204081632653061\\
0.428571428571429	0.217316017316017\\
0.418367254157114	0.224489795918367\\
0.408163265306122	0.232023191094619\\
0.391636505312927	0.244897959183673\\
0.387755102040816	0.248081242066204\\
0.367796072807964	0.26530612244898\\
0.36734693877551	0.265715316429602\\
0.346938775510204	0.285203535959838\\
0.346428868968476	0.285714285714286\\
0.327214822267331	0.306122448979592\\
0.326530612244898	0.306893552875695\\
0.309859009187079	0.326530612244898\\
0.306122448979592	0.331225726654298\\
0.294135025015389	0.346938775510204\\
0.285714285714286	0.358766233766234\\
0.279847920296032	0.36734693877551\\
0.266819307520321	0.387755102040816\\
0.26530612244898	0.390255458826887\\
0.254864026623524	0.408163265306122\\
0.244897959183673	0.426680581323439\\
0.243916129127645	0.428571428571429\\
0.233799758582267	0.448979591836735\\
0.224489795918367	0.469387755102041\\
0.224489795918367	0.469387755102041\\
0.215831296690913	0.489795918367347\\
0.20781635184899	0.510204081632653\\
0.204081632653061	0.520261595547309\\
0.200350736165414	0.530612244897959\\
0.193380552055179	0.551020408163265\\
0.186879318433636	0.571428571428571\\
0.183673469387755	0.582024324881468\\
0.180783629784166	0.591836734693878\\
0.175055392068987	0.612244897959184\\
0.169679173103707	0.63265306122449\\
0.164623491177351	0.653061224489796\\
0.163265306122449	0.658759276437848\\
0.159841941232795	0.673469387755102\\
0.155323623854892	0.693877551020408\\
0.15105381572875	0.714285714285714\\
0.147012563416438	0.73469387755102\\
0.143181993947522	0.755102040816326\\
0.142857142857143	0.756883116883117\\
0.139530212601834	0.775510204081633\\
0.136058680880415	0.795918367346939\\
0.132755745217897	0.816326530612245\\
0.129609414897736	0.836734693877551\\
0.126608809926791	0.857142857142857\\
0.123744035335684	0.877551020408163\\
0.122448979591837	0.887090290661719\\
0.121000064856769	0.897959183673469\\
0.118370126653263	0.918367346938776\\
0.115852100959592	0.938775510204082\\
0.113438993228845	0.959183673469388\\
0.111124379898736	0.979591836734694\\
0.108902351293022	1\\
};
\end{axis}

\begin{axis}[%
width=1in,
height=1in,
at={(0in,0in)},
scale only axis,
point meta min=0,
point meta max=1,
xmin=0,
xmax=1,
ymin=0,
ymax=1,
axis line style={draw=none},
ticks=none,
axis x line*=bottom,
axis y line*=left
]
\end{axis}

\end{tikzpicture}%

%% file: perc2.tikz
%
%
\begin{tikzpicture}

\begin{axis}[%
width=2in,
height=2in,
at={(0.796in,0.708in)},
scale only axis,
point meta min=10,
point meta max=90,
colormap/jet,
xmin=0,
xmax=1,
xlabel style={font=\color{white!15!black}},
xlabel={$\alpha$, concentration amplitude},
ymin=0,
ymax=1,
ylabel style={font=\color{white!15!black}},
ylabel={$\beta$, flow-rate amplitude},
axis background/.style={fill=white},
label style = {font=\normalsize},
tick label style = {font=\normalsize},
xmajorgrids,
xminorgrids,
ymajorgrids
]
\addplot[contour prepared, contour prepared format=matlab, line width=2.0pt] table[row sep=crcr] {%
10	93\\
0.108902351293022	1\\
0.111124379898736	0.979591836734694\\
0.113438993228845	0.959183673469388\\
0.115852100959593	0.938775510204082\\
0.118370126653264	0.918367346938776\\
0.121000064856769	0.897959183673469\\
0.122448979591837	0.887090290661719\\
0.123744035335685	0.877551020408163\\
0.126608809926792	0.857142857142857\\
0.129609414897736	0.836734693877551\\
0.132755745217897	0.816326530612245\\
0.136058680880416	0.795918367346939\\
0.139530212601835	0.775510204081633\\
0.142857142857143	0.756883116883117\\
0.143181993947522	0.755102040816326\\
0.147012563416439	0.73469387755102\\
0.15105381572875	0.714285714285714\\
0.155323623854892	0.693877551020408\\
0.159841941232795	0.673469387755102\\
0.163265306122449	0.658759276437848\\
0.164623491177351	0.653061224489796\\
0.169679173103707	0.63265306122449\\
0.175055392068987	0.612244897959184\\
0.180783629784166	0.591836734693878\\
0.183673469387755	0.582024324881468\\
0.186879318433636	0.571428571428571\\
0.193380552055179	0.551020408163265\\
0.200350736165415	0.530612244897959\\
0.204081632653061	0.52026159554731\\
0.207816351848991	0.510204081632653\\
0.215831296690913	0.489795918367347\\
0.224489795918367	0.469387755102041\\
0.224489795918367	0.469387755102041\\
0.233799758582268	0.448979591836735\\
0.243916129127645	0.428571428571429\\
0.244897959183673	0.426680581323438\\
0.254864026623524	0.408163265306122\\
0.26530612244898	0.390255458826887\\
0.266819307520321	0.387755102040816\\
0.279847920296032	0.36734693877551\\
0.285714285714286	0.358766233766234\\
0.294135025015389	0.346938775510204\\
0.306122448979592	0.331225726654298\\
0.30985900918708	0.326530612244898\\
0.326530612244898	0.306893552875696\\
0.327214822267331	0.306122448979592\\
0.346428868968476	0.285714285714286\\
0.346938775510204	0.285203535959838\\
0.36734693877551	0.265715316429602\\
0.367796072807964	0.26530612244898\\
0.387755102040816	0.248081242066204\\
0.391636505312927	0.244897959183673\\
0.408163265306122	0.23202319109462\\
0.418367254157114	0.224489795918367\\
0.428571428571429	0.217316017316017\\
0.448496685261211	0.204081632653061\\
0.448979591836735	0.203775510204082\\
0.469387755102041	0.191249495845769\\
0.482473187120356	0.183673469387755\\
0.489795918367347	0.179611162646877\\
0.510204081632653	0.168753988868275\\
0.521022395570714	0.163265306122449\\
0.530612244897959	0.158587840730698\\
0.551020408163265	0.149035937607366\\
0.564864165329931	0.142857142857143\\
0.571428571428571	0.140032467532468\\
0.591836734693878	0.131520696052716\\
0.612244897959184	0.123451453308596\\
0.614836822566902	0.122448979591837\\
0.63265306122449	0.115781913938596\\
0.653061224489796	0.108474605751391\\
0.67183166023166	0.102040816326531\\
0.673469387755102	0.101496598639456\\
0.693877551020408	0.0948188366255593\\
0.714285714285714	0.0884155844155844\\
0.73469387755102	0.0822639661925376\\
0.736816900268051	0.0816326530612245\\
0.755102040816326	0.0763435792007221\\
0.775510204081633	0.0706361683429353\\
0.795918367346939	0.0651253508396366\\
0.810754296816584	0.0612244897959184\\
0.816326530612245	0.0597963821892393\\
0.836734693877551	0.0546359563781166\\
0.857142857142857	0.0496320346320346\\
0.877551020408163	0.0447736980627346\\
0.894587004015468	0.0408163265306122\\
0.897959183673469	0.0400510204081633\\
0.918367346938776	0.035454957740672\\
0.938775510204082	0.0309772525611035\\
0.959183673469388	0.026610350136186\\
0.979591836734694	0.0223473252937539\\
0.988986088209275	0.0204081632653061\\
1	0.0181818181818182\\
20	82\\
0.215676801332778	1\\
0.219995006873326	0.979591836734694\\
0.224489795918367	0.959183673469388\\
0.224489795918367	0.959183673469388\\
0.229153485093298	0.938775510204082\\
0.234015250003436	0.918367346938776\\
0.23908798341788	0.897959183673469\\
0.244385721843202	0.877551020408163\\
0.244897959183673	0.875624613481756\\
0.249901924927102	0.857142857142857\\
0.25567080435079	0.836734693877551\\
0.261712627880715	0.816326530612245\\
0.26530612244898	0.804629656058227\\
0.268034500975934	0.795918367346939\\
0.274651796766055	0.775510204081633\\
0.281604544578788	0.755102040816326\\
0.285714285714286	0.743506493506493\\
0.288902942313773	0.73469387755102\\
0.296568136752082	0.714285714285714\\
0.304651800280709	0.693877551020408\\
0.306122448979592	0.690280766852195\\
0.313151193986413	0.673469387755102\\
0.322130694380215	0.653061224489796\\
0.326530612244898	0.643471706864564\\
0.331612337174164	0.63265306122449\\
0.341642959569235	0.612244897959184\\
0.346938775510204	0.601946960602423\\
0.352268473260463	0.591836734693878\\
0.363542821161879	0.571428571428571\\
0.36734693877551	0.564825809111523\\
0.375512262050755	0.551020408163265\\
0.387755102040816	0.531412947954301\\
0.388268715559101	0.530612244897959\\
0.401836264610305	0.510204081632653\\
0.408163265306122	0.501152133580705\\
0.416331977617587	0.489795918367347\\
0.428571428571429	0.473593073593074\\
0.43184336705587	0.469387755102041\\
0.448461209854011	0.448979591836735\\
0.448979591836735	0.448367346938776\\
0.466280057928946	0.428571428571429\\
0.469387755102041	0.425170605791724\\
0.485446232773536	0.408163265306122\\
0.489795918367347	0.403749226963513\\
0.506099515019483	0.387755102040816\\
0.510204081632653	0.383890166975881\\
0.528399829933993	0.36734693877551\\
0.530612244897959	0.365413158270301\\
0.551020408163265	0.348164639593211\\
0.552517806777557	0.346938775510204\\
0.571428571428571	0.332012987012987\\
0.578639456414105	0.326530612244898\\
0.591836734693878	0.316844731623057\\
0.607025605817456	0.306122448979592\\
0.612244897959184	0.302561533704391\\
0.63265306122449	0.289077742533964\\
0.637907131959287	0.285714285714286\\
0.653061224489796	0.276318413729128\\
0.671581272309505	0.26530612244898\\
0.673469387755102	0.26421768707483\\
0.693877551020408	0.25271745061661\\
0.708327033112342	0.244897959183673\\
0.714285714285714	0.241766233766234\\
0.73469387755102	0.231318284889713\\
0.74852720279158	0.224489795918367\\
0.755102040816326	0.221332798475656\\
0.775510204081633	0.211773264329655\\
0.792562082731039	0.204081632653061\\
0.795918367346939	0.202606916892631\\
0.816326530612245	0.19380426716141\\
0.836734693877551	0.185338703108738\\
0.840822942977848	0.183673469387755\\
0.857142857142857	0.177186147186147\\
0.877551020408163	0.16932476161712\\
0.893766619566548	0.163265306122449\\
0.897959183673469	0.161734693877551\\
0.918367346938776	0.154397856112142\\
0.938775510204082	0.147297733322578\\
0.951848234188443	0.142857142857143\\
0.959183673469388	0.140419216042316\\
0.979591836734694	0.133748453927025\\
1	0.127272727272727\\
30	72\\
0.320380144636704	1\\
0.326530612244898	0.980049860853432\\
0.32667403676235	0.979591836734694\\
0.33319528458035	0.959183673469388\\
0.339982721022701	0.938775510204082\\
0.346938775510204	0.918690385245007\\
0.347052514177095	0.918367346938776\\
0.35439274712986	0.897959183673469\\
0.362050887232655	0.877551020408163\\
0.36734693877551	0.863936301793444\\
0.370036253407295	0.857142857142857\\
0.378357963225002	0.836734693877551\\
0.387063518567228	0.816326530612245\\
0.387755102040816	0.814744653842398\\
0.396140937305649	0.795918367346939\\
0.405652093077328	0.775510204081633\\
0.408163265306122	0.77028107606679\\
0.415596178762939	0.755102040816326\\
0.426028417862646	0.73469387755102\\
0.428571428571429	0.72987012987013\\
0.436956635811167	0.714285714285714\\
0.448448197375086	0.693877551020408\\
0.448979591836735	0.692959183673469\\
0.460500813163984	0.673469387755102\\
0.469387755102041	0.659091715737679\\
0.473197632904196	0.653061224489796\\
0.486564530060686	0.63265306122449\\
0.489795918367347	0.627887291280149\\
0.500648981905227	0.612244897959184\\
0.510204081632653	0.599026345083488\\
0.515525013228909	0.591836734693878\\
0.530612244897959	0.572238475809904\\
0.531250959763676	0.571428571428571\\
0.547871541623424	0.551020408163265\\
0.551020408163265	0.547293341579056\\
0.565478343563995	0.530612244897959\\
0.571428571428571	0.523993506493507\\
0.584154675956652	0.510204081632653\\
0.591836734693878	0.502168767193398\\
0.603988760937946	0.489795918367347\\
0.612244897959184	0.481671614100185\\
0.625078880164739	0.469387755102041\\
0.63265306122449	0.462373571129331\\
0.647534868538818	0.448979591836735\\
0.653061224489796	0.444162221706865\\
0.671479882190728	0.428571428571429\\
0.673469387755102	0.426938775510204\\
0.693877551020408	0.410616064607661\\
0.697028552389177	0.408163265306122\\
0.714285714285714	0.395116883116883\\
0.724330489688717	0.387755102040816\\
0.73469387755102	0.380372603586889\\
0.753574751059389	0.36734693877551\\
0.755102040816326	0.366322017750589\\
0.775510204081633	0.352910360316375\\
0.784881849491505	0.346938775510204\\
0.795918367346939	0.340088482945626\\
0.816326530612245	0.327812152133581\\
0.818499510985634	0.326530612244898\\
0.836734693877551	0.316041449839359\\
0.85459292682679	0.306122448979592\\
0.857142857142857	0.30474025974026\\
0.877551020408163	0.293875825171506\\
0.893396626319864	0.285714285714286\\
0.897959183673469	0.283418367346939\\
0.918367346938776	0.273340754483612\\
0.935166742973639	0.26530612244898\\
0.938775510204082	0.263618214084053\\
0.959183673469388	0.254228081948447\\
0.979591836734694	0.245149582560297\\
0.980164531677868	0.244897959183673\\
1	0.236363636363636\\
40	62\\
0.423069175722237	1\\
0.428571428571429	0.986147186147186\\
0.43121511191085	0.979591836734694\\
0.439659069666481	0.959183673469388\\
0.44844144304132	0.938775510204082\\
0.448979591836735	0.937551020408163\\
0.457547386381263	0.918367346938776\\
0.467029804222143	0.897959183673469\\
0.469387755102041	0.893012825683633\\
0.476882649652582	0.877551020408163\\
0.487151048679594	0.857142857142857\\
0.489795918367347	0.852025355596784\\
0.497837039755259	0.836734693877551\\
0.508991769726657	0.816326530612245\\
0.510204081632653	0.814162523191095\\
0.52061133066425	0.795918367346939\\
0.530612244897959	0.779063793349507\\
0.532760030337407	0.775510204081633\\
0.545441117568252	0.755102040816326\\
0.551020408163265	0.746422043564901\\
0.558705432784023	0.73469387755102\\
0.571428571428571	0.715974025974026\\
0.572598787432558	0.714285714285714\\
0.587135287281123	0.693877551020408\\
0.591836734693878	0.687492802763739\\
0.602377381435418	0.673469387755102\\
0.612244897959184	0.66078169449598\\
0.618376661200732	0.653061224489796\\
0.63265306122449	0.635669399724699\\
0.635182765496083	0.63265306122449\\
0.65284859367528	0.612244897959184\\
0.653061224489796	0.612006029684601\\
0.671424991049051	0.591836734693878\\
0.673469387755102	0.589659863945578\\
0.690991476653237	0.571428571428571\\
0.693877551020408	0.568514678598712\\
0.711619858600667	0.551020408163265\\
0.714285714285714	0.548467532467532\\
0.733389143181811	0.530612244897959\\
0.73469387755102	0.529426922284065\\
0.755102040816326	0.511311237025523\\
0.756378252358698	0.510204081632653\\
0.775510204081633	0.494047456303095\\
0.780674088251268	0.489795918367347\\
0.795918367346939	0.47757004899862\\
0.80639019245086	0.469387755102041\\
0.816326530612245	0.461820037105751\\
0.833643708204024	0.448979591836735\\
0.836734693877551	0.446744196569981\\
0.857142857142857	0.432294372294372\\
0.862526787206611	0.428571428571429\\
0.877551020408163	0.418426888725892\\
0.893186010352028	0.408163265306122\\
0.897959183673469	0.405102040816327\\
0.918367346938776	0.392283652855081\\
0.92574839458151	0.387755102040816\\
0.938775510204082	0.379938694845527\\
0.959183673469388	0.368036947854577\\
0.960385237044189	0.36734693877551\\
0.979591836734694	0.356550711193568\\
0.997220256192968	0.346938775510204\\
1	0.345454545454545\\
50	51\\
0.523800689106812	1\\
0.530612244897959	0.985889110889111\\
0.533697554360826	0.979591836734694\\
0.54394896222054	0.959183673469388\\
0.551020408163265	0.945550745550746\\
0.554589241469644	0.938775510204082\\
0.565625625158663	0.918367346938776\\
0.571428571428571	0.907954545454545\\
0.577088466156936	0.897959183673469\\
0.589002053325287	0.877551020408163\\
0.591836734693878	0.87281683833408\\
0.601378595124623	0.857142857142857\\
0.612244897959184	0.839891774891775\\
0.614267301702801	0.836734693877551\\
0.627672198552827	0.816326530612245\\
0.63265306122449	0.808965228320067\\
0.641637081699232	0.795918367346939\\
0.653061224489796	0.779849837662338\\
0.65620212079602	0.775510204081633\\
0.671390578004141	0.755102040816326\\
0.673469387755102	0.752380952380952\\
0.687235710182435	0.73469387755102\\
0.693877551020408	0.726413292589763\\
0.70379154847289	0.714285714285714\\
0.714285714285714	0.701818181818182\\
0.721100212321404	0.693877551020408\\
0.73469387755102	0.678481240981241\\
0.739207248932035	0.673469387755102\\
0.755102040816326	0.656300456300456\\
0.758161998244664	0.653061224489796\\
0.775510204081633	0.635184552289815\\
0.778018005534617	0.63265306122449\\
0.795918367346939	0.615051615051615\\
0.798833489043014	0.612244897959184\\
0.816326530612245	0.595827922077922\\
0.820671871393806	0.591836734693878\\
0.836734693877551	0.577446943300602\\
0.843602385182718	0.571428571428571\\
0.857142857142857	0.559848484848485\\
0.86770076510173	0.551020408163265\\
0.877551020408163	0.542977952280278\\
0.893050041375482	0.530612244897959\\
0.897959183673469	0.526785714285714\\
0.918367346938776	0.511226551226551\\
0.919732813096753	0.510204081632653\\
0.938775510204082	0.496259175607002\\
0.947816905997377	0.489795918367347\\
0.959183673469388	0.481845813760707\\
0.977441909478886	0.469387755102041\\
0.979591836734694	0.46795183982684\\
1	0.454545454545455\\
60	41\\
0.622631479307864	1\\
0.63265306122449	0.982261056915435\\
0.634183073824009	0.979591836734694\\
0.646132122573246	0.959183673469388\\
0.653061224489796	0.947693645640074\\
0.658520725958792	0.938775510204082\\
0.671366985036112	0.918367346938776\\
0.673469387755102	0.915102040816327\\
0.684680169245949	0.897959183673469\\
0.693877551020408	0.884311906580814\\
0.698506593472433	0.877551020408163\\
0.712865210024122	0.857142857142857\\
0.714285714285714	0.855168831168831\\
0.727770592417459	0.836734693877551\\
0.73469387755102	0.827535559678417\\
0.743271955709224	0.816326530612245\\
0.755102040816326	0.80128967557539\\
0.759400319755883	0.795918367346939\\
0.775510204081633	0.776321648276535\\
0.776188912562954	0.775510204081633\\
0.793662417975057	0.755102040816326\\
0.795918367346939	0.75253318110461\\
0.811869933580194	0.73469387755102\\
0.816326530612245	0.729835807050093\\
0.830856302303764	0.714285714285714\\
0.836734693877551	0.708149690031223\\
0.850666352637952	0.693877551020408\\
0.857142857142857	0.687402597402597\\
0.871348478328982	0.673469387755102\\
0.877551020408163	0.667529015834664\\
0.892955008500296	0.653061224489796\\
0.897959183673469	0.648469387755102\\
0.915542624830204	0.63265306122449\\
0.918367346938776	0.630169449598021\\
0.938775510204082	0.612579656368476\\
0.93917061541492	0.612244897959184\\
0.959183673469388	0.595654679666838\\
0.963885548611759	0.591836734693878\\
0.979591836734694	0.579352968460111\\
0.989774827087873	0.571428571428571\\
1	0.563636363636364\\
70	31\\
0.719618340842831	1\\
0.732725171847478	0.979591836734694\\
0.73469387755102	0.976589878375593\\
0.74627510567259	0.959183673469388\\
0.755102040816326	0.946278894850323\\
0.760310207274331	0.938775510204082\\
0.774849810745705	0.918367346938776\\
0.775510204081633	0.917458744263256\\
0.789899893271146	0.897959183673469\\
0.795918367346939	0.890014747157604\\
0.805507383032012	0.877551020408163\\
0.816326530612245	0.863843692022263\\
0.821699083083961	0.857142857142857\\
0.836734693877551	0.838852436761844\\
0.838503559542137	0.836734693877551\\
0.855945776186612	0.816326530612245\\
0.857142857142857	0.81495670995671\\
0.874058466703153	0.795918367346939\\
0.877551020408163	0.79208007938905\\
0.892884843621389	0.775510204081633\\
0.897959183673469	0.77015306122449\\
0.912462586215558	0.755102040816326\\
0.918367346938776	0.749112347969491\\
0.932832124350065	0.73469387755102\\
0.938775510204082	0.728900137129951\\
0.954036901273298	0.714285714285714\\
0.959183673469388	0.709463545572968\\
0.976123667098919	0.693877551020408\\
0.979591836734694	0.690754097093383\\
0.999142807237076	0.673469387755102\\
1	0.672727272727273\\
80	21\\
0.814812010147287	1\\
0.816326530612245	0.997851576994434\\
0.829378902515349	0.979591836734694\\
0.836734693877551	0.969555183492466\\
0.844443903031278	0.959183673469388\\
0.857142857142857	0.942510822510823\\
0.860028902866534	0.938775510204082\\
0.876151140186564	0.918367346938776\\
0.877551020408163	0.916631142943435\\
0.892830915001095	0.897959183673469\\
0.897959183673469	0.891836734693878\\
0.910106749379132	0.877551020408163\\
0.918367346938776	0.868055246340961\\
0.928006694388141	0.857142857142857\\
0.938775510204082	0.845220617891425\\
0.946560620323227	0.836734693877551\\
0.959183673469388	0.823272411479098\\
0.96580037229958	0.816326530612245\\
0.979591836734694	0.802155225726654\\
0.985759942069542	0.795918367346939\\
1	0.781818181818182\\
90	11\\
0.908246563931695	1\\
0.918367346938776	0.98699814471243\\
0.924211345073377	0.979591836734694\\
0.938775510204082	0.9615410986529\\
0.940703981569285	0.959183673469388\\
0.957741091619854	0.938775510204082\\
0.959183673469388	0.937081277385229\\
0.975346799219742	0.918367346938776\\
0.979591836734694	0.913556354359926\\
0.993554889012085	0.897959183673469\\
1	0.890909090909091\\
};
\end{axis}

\begin{axis}[%
width=1in,
height=1in,
at={(0in,0in)},
scale only axis,
point meta min=0,
point meta max=1,
xmin=0,
xmax=1,
ymin=0,
ymax=1,
axis line style={draw=none},
ticks=none,
axis x line*=bottom,
axis y line*=left
]
\end{axis}
\end{tikzpicture}%

%% file: main.bbl
\begin{thebibliography}{plain}

\bibitem{BC16}
Bastin, G., Coron, J.-M: Stability and Boundary Stabilization of 1-D Hyperbolic Systems, Basel: Birkh\"auser (2016)

\bibitem{BSZ19}
Benner, P., Seidel-Morgenstern, A., Zuyev, A.: Periodic switching strategies for an isoperimetric control problem with application to nonlinear chemical reactions. Appl. Math. Model.	69, 287-300 (2019)

\bibitem{Kruzhkov} 
Chechkin, G.A.,  Goritsky, A.Yu.: S.N. Kruzhkov's lectures on first-order quasilinear PDEs, E. Emmrich and P. Wittbold. De Gruyter Proceedings in Mathematics, De Gruyter, 1-68, Analytical and Numerical Aspects of Partial Differential Equations (2009)

\bibitem{D67}
Douglas, J. M.: Periodic reactor operation. Ind. Eng. Chem. Process Des. Dev. 6(1), 43-48 (1967) 

\bibitem{F21}
Felischak, M., Kaps, L., Hamel, C., Nikolic, D., Petkovska, M., Seidel-Morgenstern, A.: Analysis and experimental demonstration of forced periodic operation of an adiabatic stirred tank reactor: Simultaneous modulation of inlet concentration and total flow-rate. Chem. Eng. J. 410, 128197 (2021)

\bibitem{Grab1985} 
Grabm\"uller, H., Hoffmann, U., Sch\"adlich, K.: Prediction of conversion improvements by periodic operation for isothermal plug-flow reactors. Chem. Eng. Sci. 40, 951-960 (1985)

\bibitem{MS-MP2008}  
Markovi\'c, A.,  Seidel-Morgenstern, A.,  Petkovska, M.: Evaluation of the potential of periodically operate reactors based on the second order frequency response function. Chem. Eng. Res. Des. 86, 682--691  (2008)

\bibitem{N22}
Nikoli\'c, D., Seidel, C., Felischak, M., Mili\v{c}i\'c, T., Kienle, A., Seidel-Morgenstern,~A., Petkovska,~M.: Forced periodic operations of a chemical reactor for methanol synthesis --- The search for the best scenario based on Nonlinear Frequency Response Method. Part I Single input modulations. Chem. Eng. Sci. 248, 117134 (2022)


\bibitem{NS97}
Nitka-Stycze\'n, K.: The heredity shift operator in descent approach to the optimal periodic hereditary control problem. Int. J. Syst. Sci. 28(7), 705-719 (1997)

\bibitem{SH13}
Petkovska,~M., Seidel-Morgenstern,~A.: Evaluation of periodic processes, in the book by Silveston, P.L., Hudgins, R.R., Eds.: Periodic Operation of Chemical Reactors, Butterworth-Heinemann, Oxford (2013), Chapter 14, 387-413.

\bibitem{S76}
Schmitendorf W. E.: Pontryagin's principle for problems with isoperimetric constraints and for problems with inequality terminal constraints. J. Optim. Theory Appl. 18(4), 561-567 (1976)


\bibitem{S91}
Stycze\'n, K.: Optimal periodic non-boundary control for a tubular reactor: A performance improvement criterion. Syst. Control Lett. 16(1), 71-78  (1991)

\bibitem{TCR87}
Thullie, J., Chiao, L.,  Rinker, R. G.: Generalized treatment of concentration forcing in fixed-bed plug-flow reactors. Chem. Eng. Sci. 42(5), 1095-1101 (1987)

\bibitem{WWL2011}
Wang, J.-W., Wu, H.-N., Li, H.-X.: Distributed fuzzy control design of nonlinear hyperbolic PDE systems with application to nonisothermal plug-flow reactor. IEEE Trans. Fuzzy Syst. 19(3), 514--526 (2011)

\bibitem{WOM81}
Watanabe, N., Onogi, K., Matsubara, M.: Periodic control of continuous stirred tank reactors - I: The pi criterion and its applications to isothermal cases. Chem. Eng. Sci. 36(5), 809-818 (1981)
    
\bibitem{ZS-MB2017} 
Zuyev, A., Seidel-Morgenstern, A., Benner, P.: An isoperimetric optimal control problem for a non-isothermal chemical reactor with periodic inputs. Chem. Eng. Sci. 161, 206--214 (2017)


\end{thebibliography}
